\documentclass[12pt]{article}  
\usepackage[T1]{fontenc}
\usepackage{amsxtra}
\usepackage{color}
\usepackage{amscd,amsmath,amsfonts,amssymb,
	amsthm,amstext,latexsym}
\usepackage{enumerate,pifont,graphicx}
\usepackage{xfrac}
\usepackage[english]{babel}
\usepackage[all,cmtip]{xy} 

\title{Constant mean curvature $n$-noids in hyperbolic space}
\author{Thomas Raujouan}

\newtheorem{proposition}{Proposition}
\newtheorem{theorem}{Theorem}
\newtheorem{lemma}{Lemma}
\newtheorem{definition}{Definition}
\newtheorem{corollary}{Corollary}
\newtheorem{remark}{Remark}

\newtheorem{claim}{Claim}

\DeclareFontFamily{U}{mathx}{\hyphenchar\font45}
\DeclareFontShape{U}{mathx}{m}{n}{
	<5> <6> <7> <8> <9> <10>
	<10.95> <12> <14.4> <17.28> <20.74> <24.88>
	mathx10
}{}
\DeclareSymbolFont{mathx}{U}{mathx}{m}{n}
\DeclareFontSubstitution{U}{mathx}{m}{n}
\DeclareMathAccent{\widecheck}{0}{mathx}{"71}
\DeclareMathAccent{\widetilde}{0}{mathx}{"72}
\DeclareMathAccent{\widebar}{0}{mathx}{"73}
\DeclareMathAccent{\widevec}{0}{mathx}{"74}
\DeclareMathAccent{\widehat}{0}{mathx}{"70}

\newcommand{\Hh}{\mathbb{H}}
\newcommand{\Cc}{\mathbb{C}}
\newcommand{\Rr}{\mathbb{R}}
\newcommand{\Zz}{\mathbb{Z}}
\newcommand{\Ss}{\mathbb{S}}
\newcommand{\Aa}{\mathbb{A}}
\newcommand{\Nn}{\mathbb{N}}
\newcommand{\Dd}{\mathbb{D}}
\newcommand{\Bb}{\mathbb{B}}

\newcommand{\wW}{\mathcal{W}}
\newcommand{\mM}{\mathcal{M}}
\newcommand{\hH}{\mathcal{H}}
\newcommand{\dD}{\mathcal{D}}
\newcommand{\cC}{\mathcal{C}}
\newcommand{\Oo}{\mathcal{O}}
\newcommand{\lL}{\mathcal{L}}
\newcommand{\bB}{\mathcal{B}}
\newcommand{\aA}{\mathcal{A}}
\newcommand{\pP}{\mathcal{P}}

\newcommand{\tT}{\mathcal{T}}

\newcommand{\la}{\lambda}
\newcommand{\I}{\mathrm{I}}
\newcommand{\varx}{\mathbf{x}}

\newcommand{\Sigmatilde}{\widetilde{\Sigma}}

\newcommand{\Iwa}{\mathop \mathrm{Iwa}}
\newcommand{\Uni}{\mathop \mathrm{Uni}}

\newcommand{\Pos}{\mathop \mathrm{Pos}}
\newcommand{\Tub}{\mathop \mathrm{Tub}}
\newcommand{\Sym}{ \mathrm{Sym}}
\newcommand{\Nor}{\mathrm{Nor}}

\newcommand{\pol}{\mathrm{pol}}
\newcommand{\Pol}{\mathrm{Pol}}
\newcommand{\arcoth}{\mathop \mathrm{arcoth}}

\newcommand{\conj}[1]{\overline{#1}}
\newcommand{\scal}[2]{\left\langle #1, #2 \right\rangle}
\newcommand{\trace}{\mathop \mathrm{tr}}
\newcommand{\distH}[2]{d_{\Hh^3}\left(#1,#2\right)}
\newcommand{\norm}[1]{\left\Vert#1\right\Vert}
\newcommand{\function}[5]{
	\begin{array}{ccccc}
		#1 & : & #2 & \longrightarrow & #3 \\
		& & #4 & \longmapsto & #5 \\
	\end{array}
}
\renewcommand{\Re}{\mathop \mathrm{Re}}
\renewcommand{\Im}{\mathop \mathrm{Im}}
\newcommand{\abs}[1]{\left|#1\right|}
\newcommand{\Mat}{\mathrm{Mat}}
\newcommand{\tendsto}[1]{\underset{#1}{\longrightarrow}}
\newcommand{\supp}[1]{\underset{#1}{\sup}}
\newcommand{\Res}{\mathop\mathrm{Res}}
\newcommand{\Deck}{\mathrm{Deck}}
\newcommand{\geod}{\mathrm{geod}}

\renewcommand{\sl}{\mathfrak{sl}}
\newcommand{\slfrak}{\mathfrak{sl}}
\newcommand{\su}{\mathfrak{su}}
\newcommand{\SL}{\mathrm{SL}}
\newcommand{\SU}{\mathrm{SU}}
\newcommand{\LSUdeux}{\Lambda \SU(2)}
\newcommand{\LsldeuxC}{\Lambda \sl(2,\Cc)}
\newcommand{\LsldeuxCR}{\LsldeuxC_R}
\newcommand{\LsldeuxCrho}{\LsldeuxC_\rho}
\newcommand{\LSLdeuxCR}{\Lambda \SL(2,\Cc)_R}
\newcommand{\LSLdeuxCrho}{\Lambda \SL(2,\Cc)_\rho}
\newcommand{\LSLdeuxC}{\Lambda \SL(2,\Cc)}
\newcommand{\LSUdeuxR}{\Lambda \SU(2)_R}
\newcommand{\LSUdeuxrho}{\Lambda \SU(2)_\rho}
\newcommand{\LplusSLdeuxCR}{\Lambda_+\SL(2,\Cc)_R}
\newcommand{\LplusSLdeuxCrho}{\Lambda_+\SL(2,\Cc)_\rho}
\newcommand{\LplusRSLdeuxC}{\Lambda_+^\Rr\SL(2,\Cc)}

\newcommand{\LsuR}{\Lambda \su(2)_R}
\newcommand{\Lsurho}{\Lambda \su(2)_\rho}
\newcommand{\SLdeuxCplusplus}{\SL\left(2,\Cc\right)^{++}}

\newcommand{\wWRgeqzero}{\wW_R^{\geq 0}}

\begin{document}

\maketitle

\begin{abstract}
Using the DPW method, we construct genus zero Alexandrov-embedded constant mean curvature (greater than one) surfaces with any number of Delaunay ends in hyperbolic space.
\end{abstract}

\section*{Introduction}

In \cite{dpw}, Dorfmeister, Pedit and Wu introduced a loop group method (the DPW method) for constructing harmonic maps from a Riemann surface into a symmetric space. 
As a consequence, their method provides a Weierstrass-type representation of constant mean curvature surfaces (CMC) in Euclidean space $\Rr^3$, three-dimensional sphere $\Ss^3$, or hyperbolic space $\Hh^3$.
Many examples have been constructed (see for example \cite{newcmc,dw,kkrs,dik,heller1,heller2}). Among them, Traizet \cite{nnoids,minoids} showed how the DPW method in $\Rr^3$ can construct genus zero $n$-noids with Delaunay ends (as Kapouleas did with partial differential equations techniques in \cite{kapouleas})
and glue half-Delaunay ends to minimal surfaces (as did Mazzeo and Pacard in \cite{pacard}, also with PDE techniques). 
A natural question is wether these constructions can be carried out in $\Hh^3$. Although properly embedded CMC annuli of mean curvature $H>1$ in $\Hh^3$ are well-known since the work of Korevaar, Kusner, Meeks and Solomon \cite{meeks}, no construction similar to \cite{kapouleas} or \cite{pacard} can be found in the literature.
This paper uses the DPW method in $\Hh^3$ to construct these surfaces. The two resulting theorems are as follows.

\begin{theorem}\label{theoremConstructionNnoids}
	Given a point $p\in\Hh^3$, $n\geq 3$ distinct unit vectors $u_1, \cdots, u_n$ in the tangent space of $\Hh^3$ at $p$ and $n$ non-zero real weights $\tau_1,\cdots,\tau_n$ satisfying the balancing condition
	\begin{equation}\label{eqBalancingNnoids}
	\sum_{i=1}^{n}\tau_i u_i = 0
	\end{equation}
	and given $H>1$, there exists a smooth $1$-parameter family of CMC $H$ surfaces $\left(M_t\right)_{0<t<T}$ with genus zero, $n$ Delaunay ends and the following properties:
	\begin{enumerate}
		\item Denoting by $w_{i,t}$ the weight of the $i$-th Delaunay end,
		\begin{equation*}
		\lim\limits_{t\to 0} \frac{w_{i,t}}{t} = \tau_i.
		\end{equation*}
		\item Denoting by $\Delta_{i,t}$ the axis of the $i$-th Delaunay end, $\Delta_{i,t}$ converges to the oriented geodesic through the point $p$ in the direction of $u_i$.
		\item If all the weights $\tau_i$ are positive, then $M_t$ is Alexandrov-embeddedd.
		\item If all the weights $\tau_i$ are positive and if for all $i\neq j\in[1,n]$, the angle $\theta_{ij}$ between $u_i$ and $u_j$ satisfies
		\begin{equation}\label{eqAnglesNnoid}
		\left| \sin\frac{\theta_{ij}}{2} \right|>\frac{\sqrt{H^2-1}}{2H},
		\end{equation}
		then $M_t$ is embedded.
	\end{enumerate}
\end{theorem}

\begin{theorem}\label{theoremConstructionMinoids}
	Let $M_0\subset\Rr^3$ be a non-degenerate minimal $n$-noid with $n\geq 3$ and let $H>1$. There exists a smooth family of CMC $H$ surfaces $\left(M_t\right)_{0<|t|<T}$ in $\Hh^3$ such that
	\begin{enumerate}
		\item The surfaces $M_t$ have genus zero and $n$ Delaunay ends.
		\item After a suitable blow-up, $M_t$ converges to $M_0$ as $t$ tends to $0$.
		\item If $M_0$ is Alexandrov-embedded, then all the ends of $M_t$ are of unduloidal type if $t>0$ and of nodoidal type if $t<0$. Moreover, $M_t$ is Alexandrov-embedded if $t>0$.
	\end{enumerate}
\end{theorem}

Following the proofs of \cite{nnoids,minoids} gives an effective strategy to construct the desired CMC surfaces $M_t$. This is done in Sections \ref{sectionNnoids} and \ref{sectionMinoids}. However, showing that $M_t$ is Alexandrov-embedded  requires a precise knowledge of its ends. This is the purpose of the main theorem  (Section \ref{sectionPerturbedDelaunayImmersions}, Theorem \ref{theoremPerturbedDelaunay}). We consider a family of holomorphic perturbations of the data giving rise via the DPW method to a half-Delaunay embedding $f_0: \Dd^*\subset \Cc\longrightarrow \Hh^3$ and show that the perturbed induced surfaces $f_t(\Dd^*)$ are also embedded. Note that the domain on which the perturbed immersions are defined does not depend on the parameter $t$, which is stronger than $f_t$ having an embedded end, and is critical for showing that the surfaces $M_t$ are Alexandrov-embedded. The essential hypothesis on the perturbations is that they do not occasion a period problem on the domain $\Dd^*$ (which is not simply connected). The proof relies on the Fr\"obenius method for linear differential systems with regular singular points. Although this idea has been used in $\Rr^3$ by Kilian, Rossman, Schmitt \cite{krs} and \cite{raujouan}, the case of $\Hh^3$ generates two extra resonance points that are unavoidable and make their results inapplicable. Our solution is to extend the Fr\"obenius method to loop-group-valued differential systems.

\begin{figure}
	\centering
	\includegraphics[width=7cm]{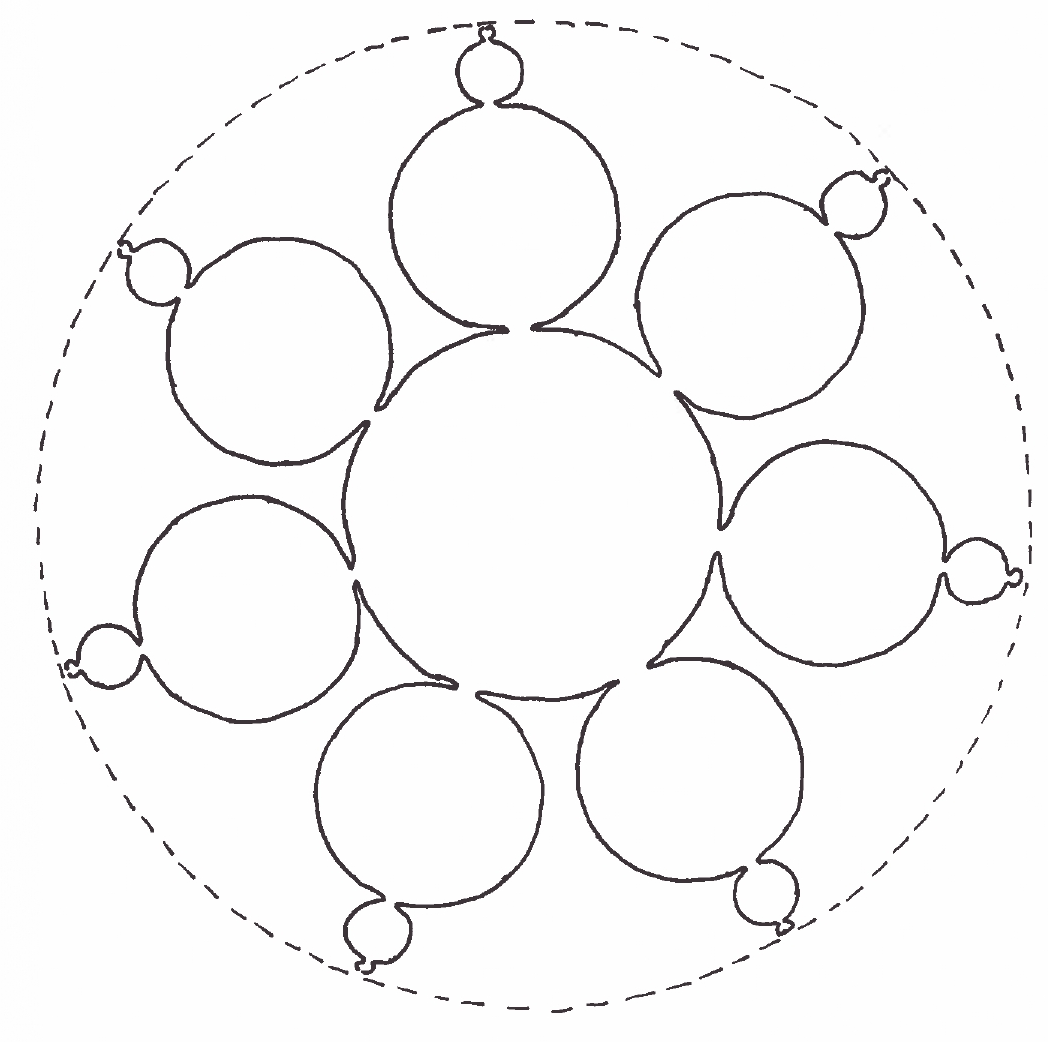}
	\caption{\textit{Theorem \ref{theoremConstructionNnoids} ensures the existence of $n$-noids with small necks. For $H>1$ small enough ($H\simeq1.5$ on the picture), there exists embedded $n$-noids with more than six ends.}}
\end{figure}

\section{Delaunay surfaces  in $\Hh^3$ via the DPW method}\label{sectionNotations}

This Section fixes the notation and recalls the DPW method in $\Hh^3$.

\subsection{Hyperbolic space}

\paragraph{Matrix model.}
Let $\Rr^{1,3}$ denote the space $\Rr^4$ with the Lorentzian metric $\scal{x}{x} = -x_0^2+x_1^2+x_2^2+x_3^2$. Hyperbolic space is the subset $\Hh^3 $ of vectors $x\in\Rr^{1,3}$ such that $\scal{x}{x} = -1$ and $x_0>0$, with the metric induced by $\Rr^{1,3}$. The DPW method constructs CMC immersions into a matrix model of $\Hh^3$. Consider the identification
\begin{equation*}
x=(x_0,x_1,x_2,x_3)\in\Rr^{1,3} \simeq X = \begin{pmatrix}
x_0+x_3 & x_1+ix_2\\
x_1-ix_2 & x_0 - x_3
\end{pmatrix}\in \hH_2 
\end{equation*}
where $\hH_2:=\{M\in \mM(2,\Cc)\mid M^*=M \}$ denotes the Hermitian matrices.
In this model, $\scal{X}{X} = -\det X$ and $\Hh^3$ is identified with the set $\hH_2^{++}\cap \SL(2,\Cc)$ of Hermitian positive definite matrices with determinant $1$. This fact enjoins us to write
\begin{equation*}
	\Hh^3 = \left\{ F{F}^* \mid F\in\SL(2,\Cc) \right\}.
\end{equation*}
Setting
\begin{equation}\label{eqPauliMatrices}
\sigma_1 = \begin{pmatrix}
0& 1\\
1 & 0
\end{pmatrix}, \qquad \sigma_2 = \begin{pmatrix}
0& i\\
-i & 0
\end{pmatrix}, \qquad \sigma_3 = \begin{pmatrix}
1& 0\\
0 & -1
\end{pmatrix},
\end{equation}
gives us an orthonormal basis $\left(\sigma_1,\sigma_2,\sigma_3\right)$ of the tangent space $T_{\I_2}\Hh^3$ of $\Hh^3$ at the identity matrix. We choose the orientation of $\Hh^3$ induced by this basis.

\paragraph{Rigid motions.}

In the matrix model of $\Hh^3$, $\SL(2,\Cc)$ acts as rigid motions: for all $p\in\Hh^3$ and $A\in\SL(2,\Cc)$, this action is denoted by $$A\cdot p := Ap{A}^*\in\Hh^3.$$ This action extends to tangent spaces: for all $v\in T_p\Hh^3$,  $A\cdot v := Av{A}^*\in T_{A\cdot p}\Hh^3$. The DPW method takes advantage of this fact and contructs immersions in $\Hh^3$ with the moving frame method.

\paragraph{Geodesics.}
Let $p\in\Hh^3$ and $v\in UT_p\Hh^3$. Define the map
\begin{equation}\label{eqExpMapGamma}
	\function{\geod(p,v)}{\Rr}{\Hh^3}{t}{p\cosh t + v \sinh t.}
\end{equation}
Then $\geod(p,v)$ is the unit speed geodesic through $p$ in the direction $v$.
The action of $\SL(2,\Cc)$ extends to oriented geodesics via:
\begin{equation*}
	A\cdot \geod(p,v) := \geod(A\cdot p, A\cdot v).
\end{equation*}

\paragraph{Parallel transport.}

Let $p,q\in\Hh^3$ and $v\in T_p\Hh^3$. We denote the result of parallel transporting $v$ from $p$ to $q$ along the geodesic of $\Hh^3$ joining $p$ to $q$ by $\Gamma_p^qv \in T_q\Hh^3$. The parallel transport of vectors from the identity matrix is easy to compute with Proposition \ref{propParallTranspI2}.

\begin{proposition}\label{propParallTranspI2}
	For all $p\in\Hh^3$ and $v\in T_{\I_2}\Hh^3$, there exists a unique $S\in\hH_2^{++}\cap\SL(2,\Cc)$ such that $p=S\cdot \I_2$. Moreover, $ \Gamma_{\I_2}^p v = S\cdot v$.
\end{proposition}
\begin{proof}
	The point $p$ is in $\Hh^3$ identified with $\hH_2^{++}\cap \SL(2,\Cc)$. Define $S$ as the unique square root of $p$ in $\hH_2^{++}\cap \SL(2,\Cc)$. Then $p= S\cdot \I_2$. Define for $t\in[0,1]$:
	\begin{equation*}
	S(t) := \exp\left(t \log S\right), \qquad \gamma(t) := S(t)\cdot \I_2, \qquad v(t) := S(t)\cdot v.
	\end{equation*}
	Then $v(t)\in T_{\gamma(t)}\Hh^3$ because
	\begin{equation*}
	\scal{v(t)}{\gamma(t)} = \scal{S(t)\cdot \I_2}{S(t)\cdot v} = \scal{\I_2}{v} = 0
	\end{equation*}
	and $S\cdot v = v(1)\in T_p\Hh^3$.
	
	Suppose that $S$ is diagonal. Then
	\begin{equation*}
	S(t) = \begin{pmatrix}
	e^{\frac{at}{2}} & 0\\
	0 & e^{\frac{-at}{2}}
	\end{pmatrix} \qquad (a\in\Rr)
	\end{equation*}
	and using Equations \eqref{eqPauliMatrices} and \eqref{eqExpMapGamma},
	\begin{equation*}
	\gamma(t) = \begin{pmatrix}
	e^{at} & 0\\
	0 & e^{-at}
	\end{pmatrix} = \geod(\I_2,\sigma_3)(at)
	\end{equation*}
	is a geodesic curve. Write $v=v^1\sigma_1+v^2\sigma_2+v^3\sigma_3$ and compute $S(t)\cdot \sigma_i$ to find
	\begin{equation*}
	v(t) = v^1\sigma_1 + v^2\sigma_2 + v^3\begin{pmatrix}
	e^{at} & 0\\
	0 & -e^{-at}
	\end{pmatrix}.
	\end{equation*}
	Compute in $\Rr^{1,3}$
	\begin{equation*}
	\frac{Dv(t)}{dt} = \left(\frac{dv(t)}{dt}\right)^T = av^3\left(\gamma(t)\right)^T = 0
	\end{equation*}
	to see that $v(t)$ is the parallel transport of $v$ along the geodesic $\gamma$.
	
	If $S$ is not diagonal, write $S = QDQ^{-1}$ where $Q\in\SU(2)$ and $D\in\hH_2^{++}\cap \SL(2,\Cc)$ is diagonal. Then,
	\begin{equation*}
	S\cdot v = Q\cdot \left( D\cdot \left( Q^{-1}\cdot v \right) \right) = Q\cdot \Gamma_{\I_2}^{D\cdot \I_2}(Q^{-1}\cdot v).
	\end{equation*}
	But for all $A\in\SL(2,\Cc)$, $p,q\in\Hh^3$ and $v\in T_p\Hh^3$,
	\begin{equation*}
	A\cdot \Gamma_p^qv = \Gamma_{A\cdot p}^{A\cdot q}A\cdot v 
	\end{equation*}
	and thus
	\begin{equation*}
	S\cdot v = \Gamma_{\I_2}^pv.
	\end{equation*}
\end{proof}

\subsection{The DPW method for CMC $H>1$ surfaces in $\Hh^3$}\label{sectionDPWMethod}

\paragraph{Loop groups.}\label{sectionLoopGroups} In the DPW method, a whole family of surfaces is constructed, depending on a spectral parameter $\la$. This parameter will always be in one of the following subsets of $\Cc$ ($\rho>1$):
\begin{equation*}
	\Ss^1 = \left\{ \la\in\Cc\mid|\la|=1 \right\}, \quad \Aa_\rho = \left\{ \la\in\Cc \mid \rho^{-1}<|\la|<\rho \right\}, \quad \Dd_\rho = \left\{ \la\in\Cc\mid |\la|<\rho \right\}.
\end{equation*}
Any smooth map $f:\Ss^1\longrightarrow \mM(2,\Cc)$ can be decomposed into its Fourier series
\begin{equation*}
f(\la) = \sum_{i\in\Zz}f_i\la^i.
\end{equation*}
Let $|\cdot|$ denote a norm on $\mM(2,\Cc)$. Fix some $\rho>1$ and consider 
\begin{equation*}
\norm{f}_\rho := \sum_{i\in\Zz} |f_i|\rho^{|i|}.
\end{equation*}
Let $G$ be a Lie group or algebra of $\mM(2,\Cc)$. We define
\begin{itemize}
	\item $\Lambda G$ as the set of smooth functions $f:\Ss^1 \longrightarrow G$.
	\item $\Lambda G_\rho\subset \Lambda G$ as the set of functions $f$ such that $\norm{f}_{\rho}$ is finite. If $G$ is a group (or an algebra) then $(\Lambda G_\rho,\norm{\cdot}_{\rho})$ is a Banach Lie group (or algebra).
	\item  $\Lambda G_\rho^{\geq 0}\subset\Lambda G_\rho$ as the set of functions $f$ such that $f_i=0$ for all $i<0$.
	\item $\Lambda_+ G_\rho\subset \Lambda G_\rho^{\geq 0}$ as the set of functions such that $f_0$ is upper-triangular.
	\item  $\Lambda_+^\Rr \SL(2,\Cc)_\rho\subset \Lambda_+ \SL(2,\Cc)_\rho$ as the set of functions that have positive elements on the diagonal.
\end{itemize}
We also define $\Lambda\Cc$ as the set of smooth maps from $\Ss^1$ to $\Cc$, and $\Lambda\Cc_\rho$ and $\Lambda \Cc_\rho^{\geq 0}$ as above.
Note that every function of $\Lambda G_\rho$ holomorphically extends to $\Aa_\rho$ and that every function of $\Lambda G_\rho^{\geq 0}$ holomorphically extends to $\Dd_\rho$.

We will use the Fr\"obenius norm on $\mM(2,\Cc)$:
\begin{equation*}
\left| A \right| := \left(\sum_{i,j} |a_{ij}|^2\right)^{\frac{1}{2}}.
\end{equation*}
Recall that this norm is sub-multiplicative. Therefore, the norm $\norm{\cdot}_{\rho}$ is sub-multiplicative. Moreover,
\begin{equation*}
\forall A \in \LSLdeuxCrho, \quad \norm{A^{-1}}_\rho = \norm{A}_\rho,
\end{equation*}
\begin{equation*}
\forall A\in\Lambda\mM(2,\Cc)_\rho,\quad \forall \la\in\Aa_{\rho},\quad \left|A(\la)\right|\leq\norm{A}_{\rho}.
\end{equation*}

The DPW method relies on the Iwasawa decomposition. The following theorem is proved in \cite{minoids}.
\begin{theorem}
	The multiplication map $\LSUdeuxrho\times \LplusRSLdeuxC_\rho \longmapsto \LSLdeuxCrho$ is a smooth diffeomorphism between Banach manifolds. Its inverse map is called ``Iwasawa decomposition'' and is denoted for $\Phi\in\LSLdeuxCrho$:
	\begin{equation*}
		\Iwa(\Phi) = \left( \Uni(\Phi), \Pos(\Phi) \right).
	\end{equation*}
\end{theorem}

\paragraph{The ingredients.}\label{sectionIngredients} Let $H>1$, $q=\arcoth H>0$ and $\rho>e^q$. The DPW method takes for input data:
\begin{itemize}
	\item A Riemann surface $\Sigma$.
	\item A holomorphic $1$-form on $\Sigma$ with values in $\LsldeuxCrho$ of the following form:
	\begin{equation*}\label{eqPotentielXialphabetagamma}
		\xi = \begin{pmatrix}
		\alpha & \la^{-1}\beta \\
		\gamma & -\alpha
		\end{pmatrix}
	\end{equation*}
	where $\alpha$, $\beta$, $\gamma$ are holomorphic $1$-forms on $\Sigma$ with values in $\Lambda\Cc_{\rho}^{\geq 0}$. The $1$-form $\xi$ is called ``the potential''.
	\item A base point $z_0\in\Sigma$.
	\item An initial condition $\phi\in\LSLdeuxCrho$.
\end{itemize}

\paragraph{The recipe.} The DPW method consists in the following steps:
\begin{enumerate}
	\item Let $\widetilde{z}_0$ be any point above $z_0$ in the universal cover $\widetilde{\Sigma}$ of $\Sigma$. Solve on $\widetilde{\Sigma}$ the following Cauchy problem:
	\begin{equation}\label{eqCauchyProblem}
		\left\{
		\begin{array}{l}
		d\Phi = \Phi \xi \\
		\Phi(\widetilde{z}_0) = \phi.
		\end{array}
		\right.
	\end{equation}
	Then $\Phi: \Sigmatilde \longrightarrow \LSLdeuxCrho$ is called ``the holomorphic frame''.
	\item Compute pointwise on $\Sigmatilde$ the Iwasawa decomposition of $\Phi$:  $$(F(z),B(z)) := \Iwa \Phi(z).$$ The unitary part $F$ of this decomposition is called ``the unitary frame''.
	\item Define $f: \Sigmatilde \longrightarrow \Hh^3$ via the Sym-Bobenko formula:
	\begin{equation*}
		f(z) = F(z,e^{-q}){F(z,e^{-q})}^*  =: \Sym_q F(z) 
	\end{equation*}
	where $F(z,\la_0):=F(z)(\la_0)$.
\end{enumerate}
Then $f$ is a CMC $H>1$ ($H=\coth q$) conformal immersion from $\widetilde{\Sigma}$ to $\Hh^3$. Its Gauss map (in the direction of the mean curvature vector) is given by
\begin{equation*}
N(z) = F(z,e^{-q})\sigma_3{F(z,e^{-q})}^* =: \Nor_qF(z)
\end{equation*}
where $\sigma_3$ is defined in \eqref{eqPauliMatrices}. The differential of $f$ is given by
\begin{equation}\label{eqdf}
df(z) = 2\sinh(q)b(z)^2 F(z,e^{-q}) \begin{pmatrix}
0 & \beta(z,0)\\
\conj{\beta(z,0)} & 0
\end{pmatrix}  F(z,e^{-q})^*
\end{equation}
where $b(z)>0$ is the upper-left entry of $B(z)\mid_{\la=0}$. The metric of $f$ is given by
\begin{equation*}
	ds_f(z) = 2\sinh(q)b(z)^2 \left|\beta(z,0)\right|
\end{equation*}
and its Hopf differential reads
\begin{equation}\label{eqHopf}
	-2\beta(z,0)\gamma(z,0)\sinh q ~ dz^2.
\end{equation}

\begin{remark}
	The results of this paper hold for any $H>1$. We thus fix now $H>1$ and $q=\arcoth H$.
\end{remark}

\paragraph{Rigid motions.} Let $H\in \LSLdeuxCrho$ and define the new holomorphic frame $\widetilde{\Phi} = H\Phi$ with unitary part $\widetilde{F}$ and induced immersion $\widetilde{f} = \Sym_q \widetilde{F}$. If $H\in\LSUdeuxrho$, then $\widetilde{F} = HF$ and $\widetilde{\Phi}$ gives rise to the same immersion as $\Phi$ up to an isometry of $\Hh^3$:
\begin{equation*}
	\widetilde{f}(z) = H(e^{-q})\cdot f(z).
\end{equation*} 
If $H\notin \LSUdeuxrho$, this transformation is called a ``dressing'' and may change the surface.

\paragraph{Gauging.}
Let $G : \Sigmatilde\longrightarrow\LplusSLdeuxCrho$ and define the new potential:
\begin{equation*}
	\widehat{\xi} = \xi\cdot G := G^{-1}\xi G + G^{-1}dG.
\end{equation*}
The potential $\widehat{\xi}$ is a DPW potential and this operation is called ``gauging''. The data $\left(\Sigma,\xi,z_0,\phi\right)$ and $\left( \Sigma,\widehat{\xi},z_0,\phi\times G(z_0) \right)$ give rise to the same immersion.

\paragraph{The monodromy problem.} Since the immersion $f$ is only defined on the universal cover $\Sigmatilde$, one might ask for conditions ensuring that it descends to a well-defined immersion on $\Sigma$.
For any deck transformation $\tau\in\Deck\left(\widetilde{\Sigma}/\Sigma\right)$, define the monodromy of $\Phi$ with respect to $\tau$ as:
\begin{equation*}
\mM_\tau(\Phi)  := \Phi(\tau(z))\Phi(z)^{-1} \in \LSLdeuxCrho.
\end{equation*}
This map is independent of $z\in\widetilde{\Sigma}$. The standard sufficient conditions for the immersion $f$ to be well-defined on $\Sigma$ is the following set of equations, called the monodromy problem in $\Hh^3$:
\begin{equation}\label{eqMonodromyProblem}
\forall \tau\in\Deck\left(\widetilde{\Sigma}/\Sigma\right),\quad
\left\{
\begin{array}{l}
\mM_\tau(\Phi)\in\LSUdeuxrho,\\
\mM_\tau(\Phi)(e^{-q}) = \pm\I_2.
\end{array}
\right.
\end{equation}
Use the point $\widetilde{z}_0$ defined in step 1 of the DPW method to identify the fundamental group $\pi_1(\Sigma,z_0)$ with $\Deck(\widetilde\Sigma / \Sigma)$. Let $\left\{ \gamma_i \right\}_{i\in I}$ be a set of generators of  $\pi_1\left(\Sigma,z_0\right)$.  Then the problem \eqref{eqMonodromyProblem} is equivalent to
\begin{equation}\label{eqMonodromyProblemLoop}
\forall i\in I, \quad 
\left\{
\begin{array}{l}
\mM_{\gamma_i}(\Phi)\in\LSUdeuxrho,\\
\mM_{\gamma_i}(\Phi)(e^{-q}) = \pm\I_2.
\end{array}
\right.
\end{equation}

\paragraph{Example: the standard sphere.}
The DPW method can produce spherical immersions of $\Sigma = \Cc\cup\left\{\infty \right\}$ with the potential
\begin{equation*}
\xi_\Ss(z,\la) = \begin{pmatrix}
0 & \la^{-1}dz \\
0 & 0
\end{pmatrix}
\end{equation*}
and initial condition $\Phi_\Ss(0,\la) = \I_2$. 
The potential is not regular at $z=\infty$ because it has a double pole there. However, the immersion will be regular at this point because $\xi_\Ss$ is gauge-equivalent to a regular potential at $z=\infty$. Indeed, consider on $\Cc^*$ the gauge
\begin{equation*}
	G(z,\la) = \begin{pmatrix}
	z&0\\
	-\la & \frac{1}{z}
	\end{pmatrix}.
\end{equation*}
The gauged potential is then
\begin{equation*}
	\xi_\Ss\cdot G (z,\la) = \begin{pmatrix}
	0 & \la^{-1}z^{-2}dz \\
	0&0
	\end{pmatrix}
\end{equation*}
which is regular at $z=\infty$. 
The holomorphic frame is
\begin{equation}\label{eqPhi_Ss}
\Phi_\Ss(z,\la) = \begin{pmatrix}
1 & \la^{-1}z \\
0 & 1
\end{pmatrix}
\end{equation}
and its unitary factor is 
\begin{equation*}
F_\Ss(z,\la) = \frac{1}{\sqrt{1+|z|^2}}\begin{pmatrix}
1 & \la^{-1}z\\
-\la \conj{z} & 1
\end{pmatrix}.
\end{equation*}
The induced CMC-$H$ immersion is
\begin{equation*}
f_\Ss(z) = \frac{1}{1+|z|^2}\begin{pmatrix}
1+ e^{2q}|z|^2 & 2z\sinh q \\
2\conj{z}\sinh q & 1+e^{-2q}|z|^2
\end{pmatrix}.
\end{equation*}
It is not easy to see that $f_\Ss(\Sigma)$ is a sphere because it is not centered at $\I_2$. To solve this problem, notice that $F_\Ss(z,e^{-q}) = R(q)\widetilde{F}_\Ss(z)R(q)^{-1}$ where
\begin{equation}\label{eqR(q)}
R(q) := \begin{pmatrix}
e^{\frac{q}{2}} & 0\\
0 & e^{\frac{-q}{2}}
\end{pmatrix}\in\SL(2,\Cc)
\end{equation}
and
\begin{equation*}
	\widetilde{F}_\Ss(z) := \frac{1}{\sqrt{1+|z|^2}}\begin{pmatrix}
	1 & z\\
	-\conj{z} & 1
	\end{pmatrix}\in\SU(2).
\end{equation*}
Apply an isometry by setting 
\begin{equation*}\label{eqftildeSs}
	\widetilde{f}_\Ss(z) := R(q)^{-1}\cdot f_\Ss(z)
\end{equation*}
and compute
\begin{equation*}
\widetilde{f}_\Ss(z) = \frac{1}{1+|z|^2}\begin{pmatrix}
e^{-q}+ e^{q}|z|^2 & 2z\sinh q \\
2\conj{z}\sinh q & e^{q} + e^{-q}|z|^2
\end{pmatrix} = (\cosh q)\I_2 + \frac{\sinh q}{1+|z|^2} \begin{pmatrix}
|z|^2-1 & 2z \\
2\conj{z} & 1-|z|^2
\end{pmatrix}
\end{equation*}
i.e.
\begin{equation}\label{eqftilde_Ss=gammav_Ss}
	\widetilde{f}_\Ss(z)= \geod(\I_2,v_\Ss(z))(q)
\end{equation}
with $\geod$ defined in \eqref{eqExpMapGamma} and where in the basis $(\sigma_1,\sigma_2,\sigma_3)$ of $T_{\I_2}\Hh^3$,
\begin{equation}\label{eqVecteurvS}
v_\Ss(z) := \frac{1}{1+|z|^2}\left( 2\Re z, 2 \Im z, |z|^2-1 \right)
\end{equation}
describes a sphere of radius one in the tangent space of $\Hh^3$ at $\I_2$ (it is the inverse stereographic projection from the north pole). Hence, $\widetilde{f}_\Ss(\Sigma)$ is a sphere centered at $\I_2$ of hyperbolic radius $q$ and $f_\Ss\left(\Sigma\right)$ is a sphere of same radius centered at $\geod(\I_2, \sigma_3)(q)$.

One can compute the normal map of $f_\Ss$: 
\begin{equation*}
	N_\Ss(z) := \Nor_q F_\Ss(z) = R(q)\cdot \widetilde{N}_\Ss(z)
\end{equation*}
where
\begin{align*}
	\widetilde{N}_\Ss(z) &:= \Nor_q\left(\widetilde{F}_\Ss(z)\right) = \frac{1}{1+|z|^2}\begin{pmatrix}
	e^{-q}- e^{q}|z|^2 & -2z\cosh q \\
	-2\conj{z}\cosh q & e^{-q}|z|^2 - e^{q}
	\end{pmatrix}\\
	&= -(\sinh q)  \I_2 - (\cosh q)  v_\Ss(z) =   -\dot{\geod}\left( \I_2,v_\Ss(z) \right)(q).
\end{align*}
Note that this implies that the normal map $\Nor_q$ is oriented by the mean curvature vector.

\subsection{Delaunay surfaces}\label{sectionDelaunaySurfaces}

\paragraph{The data.} 

Let $\Sigma=\Cc^*$, $\xi_{r,s}(z,\la) = A_{r,s}(\la)z^{-1}dz$ where
\begin{equation}\label{eqArs}
	A_{r,s}(\la) := \begin{pmatrix}
	0 & r\la^{-1}+s\\
	r\la+s & 0
	\end{pmatrix}, \qquad r,s\in\Rr, \quad \la\in\Ss^1,
\end{equation}
and initial condition $\Phi_{r,s}(1) = \I_2$. With these data, the holomorphic frame reads 
\begin{equation*}
	\Phi_{r,s}(z) = z^{A_{r,s}}.
\end{equation*}
The unitary frame $F_{r,s}$ is not explicit, but the DPW method states that the map $f_{r,s}=\Sym_q(F_{r,s})$ is a CMC $H$ immersion from the universal cover $\widetilde{\Cc^*}$ of $\Cc^*$ into $\Hh^3$.

\paragraph{Monodromy.} Computing the monodromy along $\gamma(\theta) = e^{i\theta}$ for $\theta\in [0,2\pi]$ gives
\begin{equation*}
\mM\left(\Phi_{r,s}\right):= \mM_\gamma\left(\Phi_{r,s}\right) = \exp\left(2i\pi A_{r,s}\right).
\end{equation*}
Recall that $r,s\in\Rr$ to see that $iA_{r,s}\in\Lsurho$, and thus $\mM\left(\Phi_{r,s}\right)\in\LSUdeuxrho$: the first equation of \eqref{eqMonodromyProblemLoop} is satisfied. To solve the second one, we look forward to finding $r$ and $s$ such that $A_{r,s}(e^{-q})^2=\frac{1}{4}\I_2$, which will imply that $\mM\left(\Phi_{r,s}\right)(e^{-q})=-\I_2$. This condition is equivalent to
\begin{equation}\label{eqrs}
	r^2+s^2+2rs\cosh q = \frac{1}{4}.
\end{equation}
Seeing this equation as a polynomial in $r$ and computing its discriminant ($1+4s^2\sinh^2q>0$) ensures the existence of an infinite number of solutions: given a couple $(r,s)\in\Rr^2$ solution to \eqref{eqrs}, $f_{r,s}$ is a well-defined CMC $H$ immersion from $\Cc^*$ into $\Hh^3$.

\paragraph{Surface of revolution.} 

Let $(r,s)\in\Rr^2$ satisfying \eqref{eqrs} and let $\theta\in\Rr$. Then,
\begin{equation*}
	\Phi_{r,s}\left(e^{i\theta}z\right) = \exp\left(i\theta A_{r,s}\right)\Phi_{r,s}(z).
\end{equation*}
Using $iA_{r,s}\in\Lsurho$ and diagonalising $A_{r,s}(e^{-q})$ gives
\begin{align*}
	f_{r,s}\left(e^{i\theta} z\right) &= \exp\left(i\theta A_{r,s}(e^{-q})\right)\cdot f_{r,s}(z)\\
	&=\left( H_{r,s}\exp\left(i\theta D\right)H_{r,s}^{-1}\right)\cdot f_{r,s}(z)
\end{align*}
where
\begin{equation*}
	H_{r,s} = \frac{1}{\sqrt{2}}\begin{pmatrix}
	1 & -2\left(re^{-q}+s\right)\\
	2\left(re^q+s\right) & 1
	\end{pmatrix}, \qquad D = \begin{pmatrix}
	\frac{1}{2}& 0\\0 & \frac{-1}{2}
	\end{pmatrix}.
\end{equation*}
Noting that $\exp\left(i\theta D\right)$ acts as a rotation of angle $\theta$ around the axis $\geod(\I_2,\sigma_3)$ and that $H_{r,s}$ acts as an isometry of $\Hh^3$ independant of $\theta$ shows that $\exp\left(i\theta A_{r,s}(e^{-q})\right)$ acts as a rotation around the axis $H_{r,s}\cdot \geod(\I_2,\sigma_3)$ and that $f_{r,s}$ is CMC $H>1$ immersion of revolution of $\Cc^*$ into $\Hh^3$ and by definition (as in \cite{meeks}) a Delaunay immersion.

\paragraph{The weight as a parameter.}

For a fixed $H>1$, CMC $H$ Delaunay surfaces in $\Hh^3$ form a family parametrised by the weight. This weight can be computed in the DPW framework: given a solution $(r,s)$ of \eqref{eqrs}, the weight $w$ of the Delaunay surface induced by the DPW data $(\Cc^*,\xi_{r,s},1,\I_2)$ reads
\begin{equation}\label{eqPoidsDelaunayPure}
w = 2\pi \times 4 rs\sinh q
\end{equation}
(see \cite{loopgroups} for details). 

\begin{lemma}
	Writing $t:=\frac{w}{2\pi}$ and assuming $t\neq 0$, Equations \eqref{eqrs} and \eqref{eqPoidsDelaunayPure} imply that
	\begin{equation}\label{eqSystemers}
		\left\{
		\begin{array}{l}
		t\leq T_1,\\
		r^2= \frac{1}{8}\left( 1-2Ht \pm 2 \sqrt{T_1-t}\sqrt{T_2-t} \right),\\
		s^2 = \frac{1}{8}\left( 1-2Ht \pm 2\sqrt{T_1-t}\sqrt{T_2-t} \right)
		\end{array}
		\right.
	\end{equation}
	with
	\begin{equation*}
		T_1 = \frac{\tanh\frac{q}{2}}{2}<\frac{1}{2\tanh\frac{q}{2}}=T_2.
	\end{equation*}
\end{lemma}
\begin{proof}
	First, note that \eqref{eqrs} and \eqref{eqPoidsDelaunayPure} imply
	\begin{equation*}
		r^2 + s^2 =\frac{1}{4}\left( 1 - 2t\coth q \right) = \frac{1}{4}(1-2Ht)
	\end{equation*}
	and thus 
	\begin{equation}\label{eqt<T2}
		t\leq \frac{H}{2}<T_2.
	\end{equation} 
	If $r=0$, then $t=0$. Thus, $r\neq 0$ and
	\begin{equation*}
		s = \frac{t}{4r\sinh q}.
	\end{equation*}
	Equation \eqref{eqrs} is then equivalent to
	\begin{equation*}
	r^2 + \frac{t^2}{16r^2\sinh^2 q} +\frac{Ht}{2} = \frac{1}{4} \iff r^4 - \frac{1-2Ht}{4}r^2 + \frac{t^2}{16\sinh^2q} = 0.
	\end{equation*}
	Using $\coth q = H$, the discriminant of this quadratic polynomial in $r^2$ is
	\begin{equation*}
	\Delta(t) = \frac{1}{16}\left( 1-4Ht+4t^2 \right)
	\end{equation*}
	which in turn is a quadratic polynomial in $t$ with discriminant
	\begin{equation*}
	\widetilde{\Delta} = \frac{H^2-1}{16}>0
	\end{equation*}
	because $H>1$.
	Thus,
	\begin{equation*}
	\Delta(t) = \frac{(T_1-t)(T_2-t)}{4}
	\end{equation*}
	because $H=\coth q$. Using \eqref{eqt<T2}, $\Delta(t)\geq 0$ if, and only if $t\leq T_1$ and
	\begin{equation*}
	r^2 = \frac{1}{8}\left( 1-2Ht \pm 2\sqrt{(T_1-t)(T_2-t)} \right).
	\end{equation*}
	By symmetry of Equations \eqref{eqrs} and \eqref{eqPoidsDelaunayPure}, $s^2$ is as in \eqref{eqSystemers}.
\end{proof}

We consider the two continuous parametrisations of $r$ and $s$ for $t\in(-\infty,T_1)$ such that $(r,s)$ satisfies Equations \eqref{eqrs} and \eqref{eqPoidsDelaunayPure} with $w=2\pi t$:
\begin{equation}\label{eqrst}
	\left\{
	\begin{array}{l}
	r(t) := \frac{\pm 1}{2\sqrt{2}}\left(1-2Ht + 2\sqrt{T_1-t}\sqrt{ T_2-t } \right)^{\frac{1}{2}},\\
	s(t) := \frac{t}{4r(t)\sinh q}.
	\end{array}
	\right.
\end{equation}

Choosing the parametrisation satisfying $r>s$ maps the unit circle of $\Cc^*$ onto a parallel circle of maximal radius, called a ``bulge'' of the Delaunay surface. As $t$ tends to $0$, the immersions tend towards a parametrisation of a sphere on every compact subset of $\Cc^*$, which is why we call this family of immersions ``the spherical family''. When $r<s$, the unit circle of $\Cc^*$ is mapped onto a parallel circle of minimal radius, called a ``neck'' of the Delaunay surface. As $t$ tends to $0$, the immersions degenerate into a point on every compact subset of $\Cc^*$. Nevertheless, we call this family the ``catenoidal family'' because applying a blowup to the immersions makes them converge towards a catenoidal immersion of $\Rr^3$ on every
compact subset of $\Cc^*$ (see Section \ref{sectionBlowUp} for more details). In both cases, the weight of the induced surfaces is given by $w = 2\pi t$.

\section{Perturbed Delaunay immersions}\label{sectionPerturbedDelaunayImmersions}

In this section, we study the immersions induced by a perturbation of Delaunay DPW data with small non-vanishing weights in a neighbourhood of  $z=0$. Our results are the same wether we choose the spherical or the catenoidal family of immersions. We thus drop the index $r,s$ in the Delaunay DPW data and replace it by a small value of $t=4rs\sinh q$ in a neighbourhood of $t=0$ such that
\begin{equation*}
	t<T_{\max}:=\frac{\tanh \frac{q}{2}}{2}.
\end{equation*}
For all $\epsilon>0$, we denote
\begin{equation*}
	D_\epsilon := \left\{ z\in\Cc \mid |z|<\epsilon \right\}, \qquad  D_\epsilon^* := D_\epsilon \backslash\left\{0\right\}.
\end{equation*}

\begin{definition}[Perturbed Delaunay potential]
	\label{defPerturbedDelaunayPotential}
	Let $\rho>e^q$, $0<T<T_{\max}$ and $\epsilon>0$. A perturbed Delaunay potential is a continuous one-parameter family $(\xi_t)_{t\in (-T,T)}$ of DPW potentials defined for $(t,z)\in(-T,T)\times D_{\epsilon}^*$ by
	\begin{align*}
	\xi_t(z) = A_tz^{-1}dz + C_t(z)dz
	\end{align*}
	where $A_t\in\Lambda\slfrak(2,\Cc)_{\rho}$ is a Delaunay residue as in \eqref{eqArs} satisfying \eqref{eqrst} and $C_t(z)\in\slfrak(2,\Cc)_{\rho}$ is $\cC^1$ with respect to $(t,z)$, holomorphic with respect to $z$ for all $t$ and satisfies $C_0(z)=0$ for all $z$.
\end{definition}

\begin{theorem}\label{theoremPerturbedDelaunay}
	Let $\rho>e^q$, $0<T<T_{\max}$, $\epsilon>0$ and $\xi_t$ be a perturbed Delaunay potential $\cC^2$ with respect to $(t,z)$. Let $\Phi_t$ be a holomorphic frame associated to $\xi_t$ for all $t$ via the DPW method. Suppose that the family of initial conditions $\phi_t$ is $\cC^2$ with respect to $t$, with $\phi_0=\widetilde z_0^{A_0}$, and that the monodromy problem \eqref{eqMonodromyProblemLoop} is solved for all $t\in (-T,T)$.
	Let $f_t=\Sym_q\left(\Uni \Phi_t\right)$. Then,
	\begin{enumerate}
		\item For all $\delta>0$, there exist $0<\epsilon'<\epsilon$, $T'>0$ and $C>0$ such that for all $z\in D_{\epsilon'}^*$ and $t\in(-T',T')\backslash\{0\}$,
		\begin{align*}
		\distH{f_t(z)}{f_t^\dD(z)}\leq C|t||z|^{1-\delta}
		\end{align*}
		where $f_t^\dD$ is a Delaunay immersion of weight $2\pi t$.
		\item There exist $T'>0$ and $\epsilon'>0$ such that for all $0<t<T'$, $f_t$ is an embedding of $D_{\epsilon'}^*$.
		\item The limit axis as $t$ tends to $0$ of the Delaunay immersion $f_t^\dD$ oriented towards the end at $z=0$ is given by:
		\begin{equation*}
		\frac{1}{\sqrt{2}}\begin{pmatrix}
		1 & -e^q\\
		e^{-q} & 1
		\end{pmatrix} \cdot \geod\left(\I_2,-\sigma_3\right)\quad \text{in the spherical family ($r>s$)},
		\end{equation*}
		\begin{equation*}
		\geod\left(\I_2,-\sigma_1\right)\quad \text{in the catenoidal family ($r<s$)}.
		\end{equation*}
	\end{enumerate}
\end{theorem}

Let $\xi_t$ and $\Phi_t$ as in Theorem \ref{theoremPerturbedDelaunay} with $\rho$, $T$ and $\epsilon$ fixed. This Section is dedicated to the proof of Theorem \ref{theoremPerturbedDelaunay}. 

The $\cC^2$-regularity of $\xi_t$ essentially means that $C_t(z)$ is $\cC^2$ with respect to $(t,z)$. Together with the $\cC^2$-regularity of $\phi_t$, it implies that $\Phi_t$ is $\cC^2$ with respect to $(t,z)$. Thus, $\mM(\Phi_t)$ is also $\cC^2$ with respect to $t$. These regularities and the fact that there exists a solution $\Phi_t$ solving the monodromy problem are used in Section \ref{sectionPropertyXit} to deduce an essential piece of information about the potential $\xi_t$ (Proposition \ref{propC3}).
This step then allows us to write in Section \ref{sectionzAP} the holomorphic frame $\Phi_t$ in a $Mz^AP$ form given by the Fr\"obenius method (Proposition \ref{propzAP1}), and to gauge this expression, in order to gain an order of convergence with respect to $z$ (Proposition \ref{propzAP2}). During this process, the holomorphic frame will loose one order of regularity with respect to $t$, which is why Theorem \ref{theoremPerturbedDelaunay} asks for a $\cC^2$-regularity of the data. Section \ref{sectionDressedDelaunayFrames} is devoted to the study of dressed Delaunay frames $Mz^A$ in order to ensure that the immersions $f_t$ will converge to Delaunay immersions as $t$ tends to $0$, and to estimate the growth of their unitary part around the end at $z=0$. Section \ref{sectionConvergenceImmersions} proves that these immersions do converge, which is the first point of Theorem \ref{theoremPerturbedDelaunay}. Before proving the embeddedness in Section \ref{sectionEmbeddedness}, Section \ref{sectionConvergenceNormales} is devoted to the convergence of the normal maps. Finally, we compute the limit axes in Section \ref{sectionLimitAxis}.

\subsection{A property of $\xi_t$}\label{sectionPropertyXit}

We begin by diagonalising $A_t$ in a unitary basis (Proposition \ref{propDiagA}) in order simplify the computations in Proposition \ref{propC3}, in which we use the Fr\"obenius method for a fixed value of $\la=e^{\pm q}$. This will ensure the existence of the $\cC^1$ map $P_1\in\LSLdeuxCrho$ that will be used in Section \ref{sectionzAP} to define the factor $P$ in the $Mz^AP$ form of $\Phi_t$.

\begin{proposition}\label{propDiagA}
	There exist $e^q<R<\rho$ and $0<T'<T$ such that for all $t\in(-T',T')$, $A_t = H_tD_tH_t^{-1}$ with $H_t\in\LSUdeuxR$ and $iD_t\in\LsuR$. Moreover, $H_t$ and $D_t$ are smooth with respect to $t$.
\end{proposition}
\begin{proof}
	For all $\la\in\Ss^1$,
	\begin{align}\label{eqDetA}
		-\det A_t (\la) &= \frac{1}{4} + \frac{t\la^{-1}(\la-e^q)(\la-e^{-q})}{4\sinh q}\\
		&= \frac{1}{4} + \frac{t}{4\sinh q}\left( \la+\la^{-1} - 2 \cosh q \right) \in \Rr.\nonumber
	\end{align}
	Extending this determinant as a holomorphic function on $\Aa_{{\rho}}$, there exists $T'>0$ such that
	\begin{equation*}
	\left|-\det A_t(\la) - \frac{1}{4}\right| <\frac{1}{4} \quad \forall (t,\la)\in(-T',T')\times \Aa_{{\rho}}
	\end{equation*}
	With this choice of $T'$, the function $\mu_t: \Aa_{{\rho}}\longrightarrow \Cc$ defined as the positive-real-part square root of $(-\det A_t)$ is holomorphic on $\Aa_{\rho}$ and real-valued on $\Ss^1$. Note that $\mu_t$ is also the positive-real-part eigenvalue of $A_t$ and thus $A_t=H_tD_tH_t^{-1}$ with
	\begin{equation}\label{eqHtDt}
	H_t(\la) = \frac{1}{\sqrt{2}}\begin{pmatrix}
	1 & \frac{-(r\la^{-1}+s)}{\mu_t(\la)}\\
	\frac{r\la+s}{\mu_t(\la)} & 1
	\end{pmatrix},\qquad D_t(\la) = \begin{pmatrix}
	\mu_t(\la) & 0\\
	0 & -\mu_t(\la)
	\end{pmatrix}.
	\end{equation}
	Let $e^q<R<\rho$. For all $t\in(-T',T')$, $\mu_t\in\Lambda\Cc_{R}$ and the map $t\mapsto \mu_t$ is smooth on $(-T',T')$. Moreover, $H_t\in\LSUdeux_{R}$, $iD_t\in\Lambda\su(2)_R$ and these functions are smooth with respect to $t$.
\end{proof}

\begin{remark}
	The bound $t<T'$ ensures that that $4 \det A_t(\la)$ is an integer only for $t=0$ and $\la=e^{\pm q}$. These points make the Fr\"obenius system resonant, but they are precisely the points that bear an extra piece of information due to the hypotheses on $\mM(\Phi_t)(e^q)$ and $\Phi_0$. Allowing the parameter $t$ to leave the interval $(-T',T')$ would bring other resonance points and make Section \ref{sectionzAP} invalid. This is why Theorem \ref{theoremPerturbedDelaunay} does not state that the end of the immersion $f_t$ is a Delaunay end for all $t$.
\end{remark}

\begin{remark}
	At $t=0$, the change of basis $H_t$ in the diagonalisation of $A_t$ takes different values wether $r>s$ (spherical family) or $r<s$ (catenoidal family). One has:
	\begin{equation}\label{eqH0spherical}
	H_0(\la) = \frac{1}{\sqrt{2}}\begin{pmatrix}
	1 & -\la^{-1}\\
	\la & 1
	\end{pmatrix} \qquad \text{in the spherical case},
	\end{equation}
	\begin{equation}\label{eqH0catenoidal}
	H_0(\la) = \frac{1}{\sqrt{2}}\begin{pmatrix}
	1 & -1\\
	1 & 1
	\end{pmatrix} \qquad \text{in the catenoidal case}.
	\end{equation}
	In both cases, $\mu_0 = \frac{1}{2}$, and thus, $D_0$ is the same.
\end{remark}

\paragraph{A basis of $\Lambda\mM(2,\Cc)_{\rho}$.}

Let $R$ and $T'$ given by Proposition \ref{propDiagA}. Identify $\Lambda\mM(2,\Cc)_{{\rho}}$ with the free $\Lambda\Cc_\rho$-module $\mM(2,\Lambda\Cc_{{\rho}})$ and  define for all $t\in(-T',T')$ the basis
\begin{equation*}
\bB_t = H_t\left(E_{1},E_{2},E_{3},E_{4}\right)H_t^{-1}=:\left(X_{t,1},X_{t,2},X_{t,3},X_{t,4}\right)
\end{equation*}
where 
\begin{equation*}
	E_1 = \begin{pmatrix}
	1&0\\
	0&0
	\end{pmatrix},\quad E_2 = \begin{pmatrix}
	0&1\\
	0&0
	\end{pmatrix}, \quad E_3 = \begin{pmatrix}
	0&0\\
	1&0
	\end{pmatrix}, \quad E_4 = \begin{pmatrix}
	0&0\\
	0&1
	\end{pmatrix}.
\end{equation*}
For all $t\in(-T',T')$, write
\begin{equation}\label{eqCt0}
	C_t(0) = \begin{pmatrix}
	tc_1(t) & \la^{-1}tc_2(t)\\
	tc_3(t) & -tc1(t)
	\end{pmatrix} = \sum_{j=1}^{4}t\widehat{c}_j(t)X_{t,j}.
\end{equation}
The functions $c_j, \widehat{c}_j$ are $\cC^1$ with respect to $t\in(-T',T')$ and take values in $\Lambda\Cc_{{R}}$. Moreover, the functions $c_i(t)$ holomorphically extend to $\Dd_\rho$.

\begin{proposition}\label{propC3}
	There exists a continuous function $\widetilde{c}_3: (-T',T')\longrightarrow\Lambda\Cc_{{R}}$ such that for all $\la\in\Ss^1$ and $t\in(-T',T')$,
	\begin{equation*}
		\widehat{c}_3(t) = t(\la-e^q)(\la-e^{-q})\widetilde{c}_3(t).
	\end{equation*}
\end{proposition}
\begin{proof}
	It suffices to show that $\widehat{c}_3(0)=0$ and that the holomorphic extension of $\widehat{c}_3(t)$ satisfies $\widehat{c}_3(t,e^{\pm q}) = 0$ for all $t$.
	
	To show that $\widehat{c}_3(0)=0$, recall that the monodromy problem \eqref{eqMonodromyProblemLoop} is solved for all $t$ and note that $\mM(\Phi_0) = -\I_2$, which implies that, as a function of $t$, the derivative of $\mM(\Phi_t)$ at $t=0$ is in $\Lsurho$. On the other hand, Proposition 4 in Appendix B of \cite{raujouan} ensures that
	\begin{equation*}
		\frac{d\mM(\Phi_t)}{dt}\mid_{t=0} = \left( \int_{\gamma} \Phi_0 \frac{d\xi_t}{dt}\mid_{t=0} \Phi_0^{-1} \right)\times \mM(\Phi_0)
	\end{equation*}
	where $\gamma$ is a generator of $\pi_1(D_\epsilon^*,z_0)$. Expanding  the right-hand side gives
	\begin{equation*}
		-\int_{\gamma} z^{A_0}\frac{d A_t}{dt}\mid_{t=0} z^{-A_0}z^{-1}dz - \int_{\gamma} z^{A_0}\frac{d C_t(z)}{dt}\mid_{t=0} z^{-A_0}dz \in\Lsurho.
	\end{equation*} 
	Using Proposition 4 of \cite{raujouan} once again, note that the first term is the derivative of $\mM(z^{A_t})$ at $t=0$, which is in $\Lsurho$ because $\mM(z^{A_t})\in\LSUdeuxrho$ and $\mM(z^{A_0}) = -\I_2$. Therefore, the second term is also in $\Lsurho$.
	Diagonalising $A_0$ with Proposition \ref{propDiagA} and using $H_0\in\Lambda\SU(2)_{{R}}$ gives
	\begin{equation*}
		2i\pi\Res_{z=0} \left( z^{D_0}H_0^{-1}\frac{d}{dt}C_t(z)\mid_{t=0}H_0z^{-D_0} \right)\in \Lambda\su(2)_{{R}}.
	\end{equation*}
	But using Equation \eqref{eqCt0},
	\begin{align*}
	z^{D_0}H_0^{-1}\frac{d}{dt}C_t(z)\mid_{t=0}H_0z^{-D_0} &= z^{D_0}H_0^{-1}\left(\sum_{j=1}^{4}\widehat{c}_j(0)X_{0,j}\right)H_0z^{-D_0}\\
	&= \sum_{j=1}^{4}\widehat{c}_j(0)z^{D_0}E_jz^{-D_0} \\
	&= \begin{pmatrix}
	\widehat{c}_1(0) & z\widehat{c}_2(0)\\
	z^{-1}\widehat{c}_3(0) & \widehat{c}_4(0)
	\end{pmatrix}.
	\end{align*}
	Thus,
	\begin{equation*}
		2i\pi\begin{pmatrix}
		0 & 0\\
		\widehat{c}_3(0) & 0
		\end{pmatrix}\in\Lambda\su(2)_{{R}}
	\end{equation*}
	which gives $\widehat{c}_3(0)=0$.
	
	Let $\la_0\in\{e^q,e^{-q}\}$ and $t\neq 0$. Using the Fr\"obenius method (Theorem 4.11 of \cite{teschl} and Lemma 11.4 of \cite{taylor}) at the resonant point $\la_0$ ensures the existence of $\epsilon'>0$, $B,M\in\mM(2,\Cc)$ and a holomorphic map $P: D_{\epsilon'}\longrightarrow\mM(2,\Cc)$ such that for all $z\in D_{\epsilon'}^*$,
	\begin{equation*}
		\left\{
		\begin{array}{l}
		\Phi_t(z,\la_0) = Mz^Bz^{A_t(\la_0)}P(z),\\
		B^2=0,\\
		P(0) = \I_2,\\
		\left[A_t(\la_0),d_zP(0)\right] + d_zP(0) = C_t(0,\la_0) - B.
		\end{array}
		\right.
	\end{equation*}
	Compute the monodromy of $\Phi_t$ at $\la=\la_0$:
	\begin{equation*}
		\mM(\Phi_t)(\la_0) = M\exp(2i\pi B)z^B\exp(2i\pi A_t(\la_0))z^{-B}M^{-1} = -M\exp(2i\pi B)M^{-1}.
	\end{equation*}
	Since the monodromy problem \eqref{eqMonodromyProblemLoop} is solved, this quantity equals $-\I_2$. Use $B^2=0$ to show that $B=0$ and thus,
	\begin{equation*}
	\left\{
	\begin{array}{l}
	\Phi_t(z,\la_0) = Mz^{A_t(\la_0)}P(z),\\
	P(0) = \I_2,\\
	\left[A_t(\la_0),d_zP(0)\right] + d_zP(0) = C_t(0,\la_0).
	\end{array}
	\right.
	\end{equation*}
	Diagonalise $A_t(\la_0)$ with Proposition \ref{propDiagA} and write $d_zP(0)=\sum p_jX_{t,j}$ to get for all $j\in[1,4]$
	\begin{equation*}
		p_j\left([D_t(\la_0),E_j] + E_j\right) = t\widehat{c}_j(t,\la_0)E_j.
	\end{equation*}
	In particular, using $\mu_t(\la_0)=1/2$,
	\begin{equation*}
		t\widehat{c}_3(t,\la_0) = p_3\left([D_t(\la_0),E_3] + E_3\right)=0.
	\end{equation*}
\end{proof}

Note that with the help of Equations \eqref{eqCt0} and \eqref{eqH0spherical} or \eqref{eqH0catenoidal}, and one can compute the series expansion of $\widehat{c}_3(0)$:
\begin{equation*}
	\widehat{c}_3(0) = \frac{-\la^{-1}}{2}c_2(0,0) + \Oo(\la^0) \quad \text{if}\quad r<s,
\end{equation*}
\begin{equation*}
	\widehat{c}_3(0) = \frac{1}{2} c_3(0,0) + \Oo(\la) \quad \text{if} \quad r>s.
\end{equation*}
Hence,
\begin{equation}\label{eqsc2+rc3=0}
	sc_2(t,0)+rc_3(t,0) \tendsto{t\to 0} 0.
\end{equation}

The following map will be useful in the next section:
\begin{equation}\label{eqP1}
t\in(-T',T')\longmapsto  P^1(t) := t\widehat{c}_1(t)X_{t,1} + \frac{t\widehat{c}_2(t)}{1+2\mu_t}X_{t,2} + \frac{t\widehat{c}_3(t)}{1-2\mu_t}X_{t,3} + t\widehat{c}_4(t)X_{t,4}.
\end{equation}
For all $t$, Proposition \ref{propC3} ensures that the map $P^1(t,\la)$ holomorphically extends to $\Aa_R$. Taking a smaller value of $R$ if necessary, $P^1(t)\in \Lambda\mM(2,\Cc)_R$ for all $t$. Moreover,
\begin{equation*}
\trace P^1(t) = t\widehat c_1 (t) + t\widehat c_4(t) = \trace C_t(0) = 0.
\end{equation*}
Thus, $P^1\in \cC^1((-T',T'),\LsldeuxCR)$.

\subsection{The $z^AP$ form of $\Phi_t$}\label{sectionzAP}

The map $P^1$ defined above allows us to use the Fr\"obenius method in a loop group framework and in the non-resonant case, that is, for all $t$ (Proposition \ref{propzAP1}). The techniques used in \cite{raujouan} will then apply in order to gauge the $Mz^AP$ form and gain an order on $z$ (Proposition \ref{propzAP2}).

\begin{proposition}\label{propzAP1}
	There exists $\epsilon'>0$ such that for all $t\in(-T',T')$ there exist $M_t\in\Lambda\SL(2,\Cc)_{R}$ and a holomorphic map $P_t: D_{\epsilon'}\longrightarrow \LSLdeuxCR$ such that for all $z\in D_{\epsilon'}^*$,
	\begin{equation*}
		\Phi_t(z) = M_tz^{A_t}P_t(z).
	\end{equation*}
	Moreover, $M_t$ is $\cC^1$ with respect to $t$, $M_0=\I_2$, $P_t(z)$ is $\cC^1$ with respect to $(t,z)$, $P_0(z) = \I_2$ for all $z$ and $P_t(0)=\I_2$ for all $t$.
\end{proposition}

\begin{proof}
	For all $k\in \Nn^*$ and $t\in(-T',T')$, define the linear map
	\begin{equation*}
	\function{\lL_{t,k}}{\Lambda\mM(2,\Cc)_{{\rho}}}{\Lambda\mM(2,\Cc)_{{\rho}}}{X}{[A_t,X] + kX.}
	\end{equation*} 
	Use the bases $\bB_t$ and restrict $\lL_{t,k}$ to $\Lambda\mM(2,\Cc)_R$ to get
	\begin{equation*}\label{eqDiagLlk}
	\Mat_{\bB_t}\lL_{t,k} = \begin{pmatrix}
	k & 0 &0&0\\
	0&k+2\mu_t&0&0\\
	0&0&k-2\mu_t&0\\
	0&0&0&k
	\end{pmatrix}.
	\end{equation*}
	Note that
	\begin{equation*}
	\det \lL_{t,k} = k^2(k^2-4\mu_t^2).
	\end{equation*}
	Thus, for all $k\geq 2$, $\det \lL_{t,k}$ is an invertible element of $\Lambda\Cc_{{R}}$, which implies that $\lL_{t,k}$ is invertible for all $t\in(-T',T')$ and $k\geq 2$. 
	
	Write 
	\begin{equation*}
	C_t(z) = \sum_{k\in\Nn}C_{t,k}z^k.
	\end{equation*}
	With $P^0 := \I_2$ and $P^1$ as in Equation \eqref{eqP1}, define for all $k\geq 1$:
	\begin{equation*}
	P^{k+1}(t) := \lL_{t,k+1}^{-1}\left( \sum_{i+j=k}P^i(t)C_{t,j} \right).
	\end{equation*}
	so that the sequence $(P^k)_{k\in\Nn}\subset\cC^1\left( (-T',T'),\Lambda\slfrak(2,\Cc)_{{\rho}} \right)$ satisfies the following recursive system for all $t\in(-T',T')$:
	\begin{equation*}
	\left\{
	\begin{array}{l}
	P^0(t)=\I_2,\\
	\lL_{t,k+1}(P^{k+1}(t)) = \sum_{i+j=k}P^i(t)C_{t,j}.
	\end{array}
	\right.
	\end{equation*}
	With $P_t(z):= \sum P^k(t)z^k$, the Fr\"obenius method ensures that there exists $\epsilon'>0$ such that for all $z\in D_{\epsilon'}^*$ and $t\in(-T',T')$,
	\begin{equation*}
	\Phi_t(z,\la)= M_tz^{A_t}P_t(z)
	\end{equation*}
	where $M_t\in\Lambda\SL(2,\Cc)_{{R}}$ is $\cC^1$ with respect to $t$, $P_t(z)$ is $\cC^1$ with respect to $t$ and $z$, and for all $t$, $P_t: D_{\epsilon'}\longrightarrow\Lambda\SL(2,\Cc)_{{R}}$ is holomorphic and satisfies $P_t(0)=\I_2$. Moreover, note that $P_0(z)=\I_2$ for all $z\in D_{\epsilon'}$ and thus $M_0=\I_2$.
\end{proof}

\begin{proposition}\label{propzAP2}
	There exists ${\epsilon'}>0$ such that for all $t\in(-T',T')$ there exist an admissible gauge $G_t: D_{{\epsilon}}\longrightarrow \LplusSLdeuxCR$, a change of coordinates $h_t: D_{{\epsilon'}}\longrightarrow D_{\epsilon}$, a holomorphic map $\widetilde{P}_t: D_{{\epsilon'}}\longrightarrow \LSLdeuxCR$ and $\widetilde{M}_t\in\LSLdeuxCR$ such that for all $z\in D_{{\epsilon'}}^*$,
	\begin{equation*}
		h_t^*\left( \Phi_t G_t \right)(z) = \widetilde{M}_tz^{A_t}\widetilde{P}_t(z).
	\end{equation*}
	Moreover, $\widetilde{M}_t$ is $\cC^1$ with respect to $t$, $\widetilde{M}_0=\I_2$ and there exists a uniform $C>0$ such that for all $t$ and $z$,
	\begin{equation*}
		\norm{\widetilde{P}_t(z)-\I_2}_{\rho} \leq C|t||z|^2.
	\end{equation*}
\end{proposition}

\begin{proof}
	The proof goes as in Section 3.3 of \cite{raujouan}. Expand $P^1(t)$ given by Equation \eqref{eqP1} as a series to get (this is a tedious but simple computation):
	\begin{equation*}
	P_1(t,\la) = \begin{pmatrix}
	0 & \frac{stc_2(t,0)+rtc_3(t,0)}{2s}\la^{-1} \\
	\frac{stc_2(t,0)+rtc_3(t,0)}{2r}   & 0
	\end{pmatrix} + \begin{pmatrix}
	\Oo(\la^0) & \Oo(\la^0)\\
	\Oo(\la) & \Oo(\la^0)
	\end{pmatrix}.
	\end{equation*}
	Define
	\begin{equation*}
	p_t:= {2\sinh q(sc_2(t,0)+rc_3(t,0))}
	\end{equation*}
	so that
	\begin{equation*}
	g_t := p_tA_t- P^1(t)\in\Lambda_+\slfrak(2,\Cc)_{R}
	\end{equation*}
	and recall Equation \eqref{eqsc2+rc3=0} together with $P_0 = \I_2$ to show that $g_0=0$. Thus,
	\begin{equation*}
	G_t:=\exp\left( g_tz \right)\in\Lambda_+\SL(2,\Cc)_{R}
	\end{equation*}
	is an admissible gauge.
	Let $\epsilon'<|p_t|^{-1}$ for all $t\in(-T',T')$. Define
	\begin{equation*}
	\function{h_t}{D_{\epsilon'}}{D_\epsilon}{z}{\frac{z}{1+p_tz}.}
	\end{equation*}
	Then,
	\begin{equation*}
	\widetilde{\xi}_t:=h_t^*\left(\xi_t\cdot G_t\right) = A_t z^{-1}dz + \widetilde{C}_t(z)dz
	\end{equation*}
	is a perturbed Delaunay potential as in Definition \ref{defPerturbedDelaunayPotential} such that $\widetilde{C}_t(0)=0$ for all $t\in (-T',T')$.
	The holomorphic frame
	\begin{equation*}
	\widetilde{\Phi}_t := h_t^*\left(\Phi_tG_t\right)
	\end{equation*}
	satisfies $d\widetilde{\Phi}_t = \widetilde{\Phi}_t\widetilde{\xi}_t$. With $\widetilde{C}_t(0)=0$, one can apply the Fr\"obenius method on $\widetilde{\xi}_t$ to get 
	\begin{equation*}
	\widetilde{\Phi}_t(z) = \widetilde{M}_tz^{A_t}\widetilde{P}_t(z)
	\end{equation*}
	with $\widetilde{M}_0=\I_2$ and 
	\begin{equation*}
	\norm{\widetilde{P}_t(z) - \I_2}_{{R}} \leq C|t||z|^2.
	\end{equation*}
\end{proof}

\paragraph{Conclusion.}

The new frame $\widetilde{\Phi}_t$ is associated to a perturbed Delaunay potential $(\widetilde{\xi}_t)_{t\in(-T',T')}$, defined for $z\in D_{\epsilon'}^*$, with values in $\LsldeuxCR$ and of the form
\begin{equation*}
	\widetilde{\xi}_t(z) = A_tz^{-1}dz + \widehat{C}_t(z)zdz
\end{equation*}
Note that $\widetilde{C}_t(z)\in\slfrak(2,\Cc)_{R}$ is now $\cC^1$ with respect to $(t,z)$. The monodromy problem \eqref{eqMonodromyProblemLoop} is solved for $\widetilde \Phi_t$ and for any $\widetilde z_0$ in the universal cover $\widetilde D_{\epsilon'}^*$, $\widetilde \Phi_0(\widetilde z_0) = \widetilde z_0^{A_0}$. Moreover, writing $\widetilde{f}_t:=\Sym_q (\Uni \widetilde{\Phi}_t)$ and ${f}_t:=\Sym_q (\Uni {\Phi}_t)$, then $\widetilde{f}_t=h_t^*f_t$ with $h_0(z)=z$. Hence in order to prove Theorem \ref{theoremPerturbedDelaunay} it suffices to prove the following proposition.

\begin{proposition}\label{propThmSimplifie}
	Let $\rho>e^q$, $0<T<T_{\max}$, $\epsilon>0$ and $\xi_t$ be a perturbed Delaunay potential as in Definition \ref{defPerturbedDelaunayPotential}. Let $\Phi_t$ be a holomorphic frame associated to $\xi_t$ for all $t$ via the DPW method. Suppose that the monodromy problem \eqref{eqMonodromyProblemLoop} is solved for all $t\in (-T,T)$ and that
	\begin{equation*}
		\Phi_t(z) = M_tz^{A_t}P_t(z)
	\end{equation*}
	where $M_t\in\LSLdeuxCrho$ is $\cC^1$ with respect to $t$, satisfies $M_0=\I_2$, and $P_t:D_\epsilon\longrightarrow\LSLdeuxCrho$ is a holomorphic map such that for all $t$ and $z$,
	\begin{equation*}
		\norm{P_t(z) - \I_2}_\rho \leq C|t||z|^2
	\end{equation*}
	where $C>0$ is a uniform constant.
	Let $f_t=\Sym_q\left(\Uni \Phi_t\right)$. Then the three points of Theorem \ref{theoremPerturbedDelaunay} hold for $f_t$.
\end{proposition}

Thus, we now reset the values of $\rho$, $T$ and $\epsilon$ and suppose that we are given a perturbed Delaunay frame $\xi_t$ and a holomorphic frame $\Phi_t$ associated to it and satisfying the hypotheses of Proposition \ref{propThmSimplifie}.

\subsection{Dressed Delaunay frames}\label{sectionDressedDelaunayFrames}

In this section we study dressed Delaunay frames arising from the DPW data $(\widetilde{\Cc}^*,\xi_t^\dD,1,M_t)$, where $\widetilde{\Cc}^*$ is the universal cover of $\Cc^*$ and 
\begin{equation*}
	\xi_t^\dD(z) := A_tz^{-1}dz
\end{equation*}
with $A_t$ as in \eqref{eqArs} satisfying \eqref{eqrst}, and $M_t$ as in Proposition \ref{propzAP2}. The induced holomorphic frame is
\begin{equation*}
	\Phi_t^\dD(z) = M_tz^{A_t}.
\end{equation*}
Note that the fact that the monodromy problem \eqref{eqMonodromyProblemLoop} is solved for $\Phi_t$ implies that it is solved for $\Phi_t^\dD$ because $P_t$ is holomorphic on $D_{\epsilon}$.
Let $\widetilde{D}_1^*$ be the universal cover of $D_1^*$ and let
\begin{equation*}
	F_t^\dD := \Uni \Phi_t^\dD, \qquad B_t^\dD := \Pos \Phi_t^\dD, \qquad f_t^\dD := \Sym_q F_t^\dD.
\end{equation*}
In this section, our goal is to prove the following proposition.
\begin{proposition}\label{propConvUnitFrame}
	The immersion $f_t^\dD$ is a CMC $H$ Delaunay immersion of weight $2\pi t$ for $|t|$ small enough. Moreover, for all $\delta>0$ and $e^q<R<\rho$ there exists $C,T'>0$ such that
	\begin{equation*}
	\norm{F_t^\dD(z)}_{{R}} \leq C|z|^{-\delta}.
	\end{equation*}
	for all $(t,z)\in(-T',T')\times \widetilde{D}_1^*$.
\end{proposition}

\paragraph{Delaunay immersion.}

We will need the following lemma, inspired by \cite{schmitt}.
\begin{lemma}\label{lemmaUniComPointwise}
	Let $M\in\SL(2,\Cc)$ and $\aA\in\su(2)$ such that
	\begin{equation}\label{eqLemmaPolar}
	M\exp(\aA)M^{-1}\in\SU(2).
	\end{equation}
	Then there exist $U\in\SU(2)$ and $K\in\SL(2,\Cc)$ such that $M=UK$ and $\left[ K,\aA \right]=0$.
\end{lemma}
\begin{proof}
	Let
	\begin{equation*}
	K = \sqrt{{M}^*M}, \qquad U=MK^{-1}
	\end{equation*}
	be a polar decomposition of $M$. The matrix $K$ is then hermitian and positive-definite because $\det M \neq 0$. Moreover, $U\in\SU(2)$ and Equation \eqref{eqLemmaPolar} is then equivalent to
	\begin{equation}\label{eqLemmaPolar2}
	K\exp\left( \aA \right)K^{-1} \in\SU(2).
	\end{equation}
	Let us diagonalise $K = QDQ^{-1}$ where $Q\in\SU(2)$ and 
	\begin{equation*}
	D=\begin{pmatrix}
	x&0\\0&x^{-1}
	\end{pmatrix}, \quad x>0.
	\end{equation*}
	Hence Equation \eqref{eqLemmaPolar2} now reads
	\begin{equation}\label{eqLemmaPolar3}
	D\exp\left( Q^{-1}\aA Q \right)D^{-1}\in\SU(2).
	\end{equation}
	But $Q\in\SU(2)$ and $\aA\in\su(2)$, so $Q^{-1}\aA Q \in\su(2)$ and $\exp\left( Q^{-1}\aA Q \right)\in\SU(2)$. Let us write
	\begin{equation*}
	\exp\left( Q^{-1}\aA Q \right) = \begin{pmatrix}
	p & -\conj{q} \\ q & \conj{p}
	\end{pmatrix}, \quad |p|^2+|q|^2 = 1
	\end{equation*}
	so that Equation \eqref{eqLemmaPolar3} is now equivalent to
	\begin{equation*}
	x=1 \text{ or } q=0.
	\end{equation*}
	If $x=1$ then $K=\I_2$ and $\left[ K,\aA \right]=0$. If $q=0$ then $Q^{-1}\aA Q$ is diagonal and $\aA$ is $Q$-diagonalisable. Thus $K$ and $\aA$ are simultaneously diagonalisable and $\left[ K,\aA \right]=0$.
\end{proof}

\begin{corollary}\label{corUnitCom}
	There exists $T'>0$ such that for all $t\in (-T',T')$,
	\begin{equation*}
		\Phi_t(z) = U_tz^{A_t}K_t
	\end{equation*}
	where $U_t\in\LSUdeux_{R}$ and $K_t\in \LSLdeuxC_{R}$ for any $e^q<R<\rho$.
\end{corollary}
\begin{proof}
	Write $\mM\left( \Phi_t^\dD \right) = M_t\exp\left( \aA_t \right)M_t^{-1}$ with $\aA_t := 2i\pi A_t\in \Lsurho$ continuous on $(-T,T)$. The map 
	\begin{equation*}
	M\longmapsto \sqrt{M^*M} = \exp\left(\frac{1}{2}\log M^*M\right)
	\end{equation*}
	is a diffeomorphism from a neighbourhood of $\I_2\in\LSLdeuxC_{{\rho}}$ to another neighbourhood of $\I_2$. Using the convergence of $M_t$ towards $\I_2$ as $t$ tends to $0$, this allows us to use Lemma \ref{lemmaUniComPointwise} pointwise on $\Aa_{{\rho}}$ and thus construct $K_t:=\sqrt{M^*M}\in\LSLdeuxC_{R}$ for all $t\in(-T',T')$ and any $e^q<R<\rho$. Let $U_t := M_tK_t^{-1}\in\LSLdeuxC_{R}$ and compute $U_tU_t^*$ to show that $U_t\in\LSUdeux_{R}$. Use Lemma \ref{lemmaUniComPointwise} to show that $\left[ K_t(\la), \aA_t(\la) \right]=0$ for all $\la\in\Ss^1$. Hence $\left[ K_t,\aA_t \right]=0$ and thus $\Phi_t^\dD = U_tz^{A_t}K_t$.
\end{proof}

Returning to the proof of Proposition \ref{propConvUnitFrame}, let $\theta\in\Rr$, $z\in\Cc^*$ and $e^q<R<\rho$. Apply Corollary \ref{corUnitCom} to get
\begin{equation*}
	\Phi_t^\dD(e^{i\theta}z) = U_t\exp(i\theta A_t)U_t^{-1}\Phi_t^\dD(z)
\end{equation*}
and note that $U_t\in\LSUdeuxR$, $iA_t\in\LsuR$ imply
\begin{equation}\label{eqRt}
	R_t(\theta) :=  U_t\exp(i\theta A_t)U_t^{-1} \in \LSUdeuxR.
\end{equation}
Hence
\begin{equation*}
	F_t^\dD(e^{i\theta z}) = R_t(\theta)F_t^\dD(z)
\end{equation*}
and
\begin{equation*}
f_t^\dD(e^{i\theta}z) =R_t(\theta,e^{-q})\cdot f_t^\dD(z).
\end{equation*}
Use Section \ref{sectionDelaunaySurfaces} and note that $U_t$ does not depend on $\theta$ to see that $f_t^\dD$ is a CMC immersion of revolution and hence a Delaunay immersion. Its weight can be read on its Hopf diffferential, which in turn can be read on the potential $\xi_t^\dD$ (see Equation \eqref{eqHopf}). Thus, $f_t^\dD$ is a CMC $H$ Delaunay immersion of weight $2\pi t$, which proves the first point of Proposition \ref{propConvUnitFrame}.

\paragraph{Restricting to a meridian.}

Note that for all $t\in(-T',T')$ and $z\in\Cc^*$,
\begin{equation*}\label{eqFtDmeridien}
\norm{F_t^\dD(z)}_{{R}} \leq C \norm{F_t^\dD(|z|)}_{{R}}
\end{equation*}
where
\begin{equation*}
C = \sup\left\{ \norm{R_t(\theta)}_{{R}} \mid (t,\theta)\in (-T',T')\times [0,2\pi] \right\}
\end{equation*}
depends only on $R$. We thus restrict $F_t^\dD$ to $\Rr_+^*$ with $\widehat{F}_t^\dD (x) := F_t^\dD(|z|)$ ($x=|z|$).

\paragraph{Gr\"onwall over a period.} 

Recalling the Lax Pair associated to $F_t^\dD$ (see Appendix C in \cite{raujouan}), the restricted map $\widehat{F}_t^\dD$ satisfies $d\widehat{F}_t^\dD = \widehat{F}_t^\dD \widehat{W}_tdx$ with 
\begin{equation*}
\widehat{W}_t(x,\la) = \frac{1}{x}\begin{pmatrix}
0 & \la^{-1}rb^2(x)-sb^{-2}(x)\\
sb^{-2}(x) - \la r b^2(x) & 0
\end{pmatrix}
\end{equation*}
where $b(x)$ is the upper-left entry of $B_t^\dD(x)\mid_{\la=0}$. 
Recall Section \ref{sectionDPWMethod} and define
\begin{equation*}
	g_t(x) = 2\sinh q |r|b(x)^2x^{-1}
\end{equation*}
so that the metric of $f_t^\dD$ reads $g_t(x)|dz|$. Let $\widetilde{f}_t^\dD:=\exp^*f_t^\dD$. Then the metric of $\widetilde{f}_t^\dD$ satisfies
\begin{equation*}
d\widetilde{s}^2 = 4r^2(\sinh q)^2b^4(e^u)(du^2+d\theta^2)
\end{equation*}
at the point $u+i\theta = \log z$. Using Proposition \ref{propJleliLopez} of Appendix \ref{appendixJleliLopez} gives
\begin{equation*}
\int_{0}^{ S_t}2|r|b^2(e^u)du = \pi \quad \text{and} \quad  \int_{0}^{S_t}\frac{du}{2\sinh q|r|b^2(e^u)} = \frac{\pi}{|t|}
\end{equation*}
where $S_t>0$ is the period of the profile curve of $\widetilde{f}_t$. Thus,
\begin{equation*}
\int_{1}^{e^{S_t}}|rb^2(x)x^{-1}|dx = \frac{\pi}{2} = \int_{1}^{e^{S_t}}|sb^{-2}(x)x^{-1}|dx.
\end{equation*}
Using
\begin{equation*}
\norm{\widehat{W}_t(x)}_{{R}} = \sqrt{2}\abs{sb^{-2}(x)x^{-1}} + 2R\abs{rb^2(x)x^{-1}},
\end{equation*}
we deduce
\begin{equation}\label{eqIntegraleW}
\int_{1}^{e^{S_t}}\norm{\widehat{W}_t(x)}_{R} dx = \frac{\pi}{2}(2R+\sqrt{2}) <C
\end{equation}
where $C>0$ is a constant depending only on $\rho$ and $T$.
Applying Gr\"onwall's lemma to the inequality
\begin{equation*}
\norm{\widehat{F}_t^\dD(x)}_{{R}} \leq \norm{\widehat{F}_t^\dD(1)}_{{R}} + \int_{1}^{x}\norm{\widehat{F}_t^\dD(u)}_{{R}}\norm{\widehat{W}_t(u)}_{{R}} |du|
\end{equation*}
gives
\begin{equation*}
\norm{\widehat{F}_t^\dD(x)}_{{R}} \leq \norm{\widehat{F}_t^\dD(1)}_{{R}}\exp\left( \int_{1}^{x}\norm{\widehat{W}_t(u)}_{{R}} |du| \right).
\end{equation*}
Use Equation \eqref{eqIntegraleW} together with the fact that $\widetilde{F}_0^\dD(0)=F_0^\dD(1)=\I_2$ and the continuity of Iwasawa decomposition to get $C,T>0$ such that for all $t\in(-T',T')$ and $x\in [1,e^{S_t}]$
\begin{equation}\label{eqFtDPeriode}
\norm{\widehat{F}_t^\dD(x)}_{{R}} \leq C.
\end{equation}

\paragraph{Control over the periodicity matrix.}

Let $t\in(-T',T')$ and $\Gamma_t := \widehat{F}_t^\dD(xe^{S_t})\widehat{F}_t^\dD(x)^{-1}\in\LSUdeux_{{R}}$ for all $x>0$. The periodicity matrix $\Gamma_t$ does not depend on $x$ because $\widehat{W}_t(xe^{S_t}) = \widehat{W}_t(x)$ (by periodicity of the metric in the $\log$ coordinate).  Moreover,
\begin{equation*}
\norm{\Gamma_t}_{{R}} = \norm{\widehat{F}_t^\dD(e^{S_t})\widehat{F}_t^\dD(1)^{-1}}_{{R}} \leq \norm{\widehat{F}_t^\dD(e^{S_t})}_{{R}} \norm{\widehat{F}_t^\dD(1)}_{{R}},
\end{equation*}
and using Equation \eqref{eqFtDPeriode},
\begin{equation}\label{eqPeriode}
\norm{\Gamma_t}_{{R}} \leq C.
\end{equation}

\paragraph{Conclusion.}

Let $x<1$. Then there exist $k\in\Nn^*$ and $\zeta \in \left[1,e^{S_t}\right)$ such that $x=\zeta e^{-kS_t}$. Thus, using Equations \eqref{eqFtDPeriode} and \eqref{eqPeriode},
\begin{equation*}
\norm{\widehat{F}_t^\dD(x)}_{{R}} \leq \norm{\Gamma_t^{-k}}_{{R}} \norm{\widehat{F}_t^\dD(\zeta)}_{{R}} \leq C^{k+1}.
\end{equation*}
Writing 
\begin{equation*}
k = \frac{\log \zeta}{S_t}-\frac{\log x}{S_t},
\end{equation*}
one gets
\begin{equation*}
C^k = \exp\left(\frac{\log \zeta}{S_t}\log C\right)\exp\left( \frac{-\log C}{S_t}\log x \right) \leq Cx^{-\delta_t}
\end{equation*}
where $\delta_t = \frac{\log C}{S_t}$ does not depend on $x$ and tends to $0$ as $t$ tends to $0$ (because $S_t\tendsto{t\to 0}+\infty$). Returning back to $F_t^\dD$, we showed that for all $\delta>0$ there exist $T'>0$ and $C>0$  such that for all $t\in(-T',T')$ and $0<|z|<1$,
\begin{equation*}
\norm{F_t^\dD(z)}_{{R}} \leq C|z|^{-\delta}
\end{equation*}
and Proposition \ref{propConvUnitFrame} is proved.

\subsection{Convergence of the immersions}\label{sectionConvergenceImmersions}
In this section, we prove the first point of Theorem \ref{theoremPerturbedDelaunay}: the convergence of the immersions $f_t$ towards the immersions $f_t^\dD$. Our proof relies on the Iwasawa decomposition being a diffeomorphism in a neighbourhood of $\I_2$.

\paragraph{Behaviour of the Delaunay positive part.}
Let $z\in\widetilde{D}_1^*$. The Delaunay positive part satisfies
\begin{equation*}
	\norm{B_t^\dD(z)}_{{\rho}} = \norm{F_t^\dD(z)^{-1}M_tz^{A_t}}_{{\rho}} \leq \norm{F_t^\dD(z)}_{{\rho}} \norm{M_t}_{{\rho}} \norm{z^{A_t}}_{{\rho}}.
\end{equation*}
Diagonalise $A_t=H_tD_tH_t^{-1}$ as in Proposition \ref{propDiagA}. 
By continuity of $H_t$ and $M_t$, and according to Proposition \ref{propConvUnitFrame}, there exists $C,T'>0$ such that for all $t\in(-T',T')$
\begin{equation*}
	\norm{B_t^\dD(z)}_{{R}} \leq C|z|^{-\delta}\norm{z^{-\mu_t}}_{{R}}.
\end{equation*}
Recall Equation \eqref{eqDetA} and extend $\mu_t^2=-\det A_t$ to $\Aa_{\widetilde{R}}$ with $\rho>\widetilde{R}>R$. One can thus assume that for $t\in(-T',T')$ and $\la\in\Aa_{\widetilde{R}}$,
\begin{equation*}
	\left|\mu_t(\la)\right|<\frac{1}{2}+\delta,
\end{equation*}
which implies that
\begin{equation*}
	\left|z^{-\mu_t(\la)}\right| \leq |z|^{\frac{-1}{2}-\delta}.
\end{equation*}
This gives us the following estimate in the $\Lambda\Cc_{{R}}$ norm (using Cauchy formula and $\widetilde{R}>R$):
\begin{equation*}
	\norm{z^{-\mu_t}}_{{R}}\leq C|z|^{\frac{-1}{2}-\delta}
\end{equation*}
and
\begin{equation}\label{eqBtD}
	\norm{B_t^\dD(z)}_{{R}} \leq  C|z|^{\frac{-1}{2}-2\delta}.
\end{equation}

\paragraph{Behaviour of a holomorphic frame.} Let 
\begin{equation*}
	\widetilde{\Phi}_t := B_t^\dD \left(\Phi_t^\dD\right)^{-1}\Phi_t \left(B_t^\dD\right)^{-1}.
\end{equation*}
Recall Proposition \ref{propzAP2} and use Equation \eqref{eqBtD} to get for all $t\in(-T',T')$ and $z\in D_\epsilon^*$:
\begin{equation*}
\norm{\widetilde{\Phi}_t(z) - \I_2}_{{R}} = \norm{B_t^\dD(z)\left( P_t(z) - \I_2 \right)B_t^\dD(z)^{-1}}_{{R}} \leq C|t||z|^{1-4\delta}. 
\end{equation*}

\paragraph{Behaviour of the perturbed immersion.}

Note that
\begin{align*}
\widetilde{\Phi}_t &= B_t^\dD \left(\Phi_t^\dD\right)^{-1}\Phi_t \left(B_t^\dD\right)^{-1} \\
&= \left(F_t^\dD\right)^{-1}F_t\times B_t\left(B_t^\dD\right)^{-1}
\end{align*}
and recall that Iwasawa decomposition is differentiable at the identity to get
\begin{equation*}
\norm{F_t^\dD(z)^{-1}F_t(z) - \I_2}_{{R}} = \norm{\Uni \widetilde{\Phi}_t(z) - \Uni \I_2}_{{R}} \leq C|t||z|^{1-4\delta}.
\end{equation*}
The map
\begin{equation*}
\widetilde{F}_t(z) := F_t^\dD(z,e^{-q})^{-1}F_t(z,e^{-q}) \in\SL(2,\Cc)
\end{equation*}
satisfies
\begin{equation}\label{eqFtilde-I2}
\left|\widetilde{F}_t(z) - \I_2 \right| \leq \norm{F_t^\dD(z)^{-1}F_t(z) - \I_2}_{{R}} \leq C|t||z|^{1-4\delta}.
\end{equation}

\begin{lemma}\label{lemmaTraceAA*}
	There exists a neighbourhood $V\subset \SL(2,\Cc)$ of $\I_2$ and $C>0$ such that for all $A\in\SL(2,\Cc)$,
	\begin{equation*}
		A\in V \implies \left\vert \trace\left(A{A}^*\right) - 2 \right\vert \leq C\left|A-\I_2\right|^2.
	\end{equation*}
\end{lemma}
\begin{proof}
	Consider $\exp: U\subset\slfrak(2,\Cc)  \longrightarrow V\subset \SL(2,\Cc)$ as a local chart of $\SL(2,\Cc)$ around $\I_2$. Let $A\in V$. Write
	\begin{equation*}
		\function{f}{\SL(2,\Cc)}{\Rr}{X}{\trace \left(X{X}^*\right)}
	\end{equation*}
	and $a=\log A\in U$ to get
	\begin{equation*}
		\left\vert f(A) - f(\I_2) \right\vert \leq df(\I_2)\cdot a + C|a|^2.
	\end{equation*}
	Notice that for all $a\in\slfrak(2,\Cc)$,
	\begin{equation*}
		df(\I_2)\cdot a = \trace(a+{a}^*) = 0
	\end{equation*}
	to end the proof.
\end{proof}

\begin{corollary}\label{corollarydH3}
	There exists a neighbourhood $V\subset \SL(2,\Cc)$ of $\I_2$ and $C>0$ such that for all $F_1,F_2\in\SL(2,\Cc)$, 
	\begin{equation*}
	F_2^{-1}F_1 \in V \implies  \distH{f_1}{f_2}< C\left| F_2^{-1}F_1 - \I_2 \right|,
	\end{equation*}
	where $f_i = F_i\cdot \I_2 \in\Hh^3$.
\end{corollary}
\begin{proof}
	Just remark that
	\begin{equation*}
	\distH{f_1}{f_2} = \cosh^{-1}\left( -\scal{f_1}{f_2} \right) = \cosh^{-1}\left( \frac{1}{2}\trace(f_2^{-1}f_1) \right)
	\end{equation*}
	and that
	\begin{equation*}
	\trace \left( f_2^{-1}f_1 \right) = \trace\left( \left(F_2{F_2}^*\right)^{-1}F_1{F_1}^* \right) = \trace  \widetilde{F}{\widetilde{F}}^* 
	\end{equation*}
	where $\widetilde{F} = F_2^{-1}F_1$. Apply Lemma \ref{lemmaTraceAA*} and use $\cosh^{-1}(1+x)\sim \sqrt{2x} $ as $x \to 0$ to end the proof.
\end{proof}

Without loss of generality, we can suppose from \eqref{eqFtilde-I2} that $C|t||z|^{1-4\delta}$ is small enough for $\widetilde{F}_t(z)$ to be in $V$ for all $t$ and $z$. Apply Corollary \ref{corollarydH3} to end the proof of the first point in Theorem \ref{theoremPerturbedDelaunay}:
\begin{equation*}\label{eqDistftDft}
	\distH{f_t(z)}{f_t^\dD(z)}\leq C|t||z|^{1-4\delta}.
\end{equation*}

\subsection{Convergence of the normal maps}\label{sectionConvergenceNormales}

Before starting the proof of the second point of Theorem \ref{theoremPerturbedDelaunay}, we will need to compare the normal maps of our immersions.
Let $N_t := \Nor_qF_t$ and $N_t^\dD := \Nor_qF_t^\dD$ be the normal maps associated to the immersions $f_t$ and $f_t^\dD$. This section is devoted to the proof of the following proposition.
\begin{proposition}\label{propConvNormales}
	For all $\delta>0$ there exist $\epsilon',T',C>0$ such that for all $t\in(-T',T')$ and $z\in D_\epsilon^*$,
	\begin{equation*}
	\norm{\Gamma_{f_t(z)}^{f_t^\dD(z)}N_t(z) - N_t^\dD(z)}_{T\Hh^3} \leq C|t||z|^{1-\delta}.
	\end{equation*}
\end{proposition}

The following lemma measures the lack of euclideanity in the parallel transportation of unitary vectors.
\begin{lemma}\label{lemmaTriangIneq}
	Let $a,b,c\in\Hh^3$, $v_a\in T_a\Hh^3$ and $v_b\in T_b\Hh^3$ both unitary. Let $\aA$ be the hyperbolic area of the triangle $(a,b,c)$. Then
	\begin{equation*}
		\norm{\Gamma_a^bv_a - v_b}_{T_b\Hh^3} \leq \aA + \norm{\Gamma_a^cv_a - \Gamma_b^cv_b}_{T_c\Hh^3}.
	\end{equation*}
\end{lemma}
\begin{proof}
	Just use the triangular inequality and Gauss-Bonnet formula in $\Hh^2$ to write:
	\begin{align*}
		\norm{\Gamma_a^bv_a - v_b}_{T_b\Hh^3} &= \norm{\Gamma_c^a \Gamma _b^c \Gamma_a^b v_a - \Gamma_c^a \Gamma_b^c v_b}_{T_a\Hh^3} \\
		&\leq \norm{\Gamma_c^a \Gamma _b^c \Gamma_a^b v_a - v_a}_{T_a\Hh^3}  + \norm{v_a - \Gamma_c^a \Gamma_b^c v_b}_{T_a\Hh^3} \\
		&\leq \aA + \norm{\Gamma_a^cv_a - \Gamma_b^cv_b}_{T_c\Hh^3}.
	\end{align*}
\end{proof}

Lemma \ref{lemmaPolarSymNor} below clarifies how the unitary frame encodes the immersion and the normal map.

\begin{lemma}\label{lemmaPolarSymNor}
	Let $f = \Sym_q F$ and $N=\Nor_qF$. Denoting by $(S(z),Q(z))\in\hH_2^{++}\cap \SL(2,\Cc) \times \SU(2)$ the polar decomposition of $F(z,e^{-q})$,
	\begin{equation*}
		f = S^2 \quad \text{and} \quad N = \Gamma_{\I_2}^{f}(Q\cdot \sigma_3).
	\end{equation*}
\end{lemma}
\begin{proof}
	The formula for $f$ is straightforward after noticing that $QQ^*=\I_2$ and $S^*=S$. The formula for $N$ is a direct consequence of Proposition \ref{propParallTranspI2}.
\end{proof}

\paragraph{Proof of Proposition \ref{propConvNormales}.}
Let $\delta>0$, $t\in(-T',T')$ and $z\in D_{\epsilon'}^*$. Using Lemma \ref{lemmaTriangIneq},
\begin{equation*}
	\norm{\Gamma_{f_t(z)}^{f_t^\dD(z)}N_t(z) - N_t^\dD(z)} \leq \aA + \norm{\Gamma_{f_t(z)}^{\I_2}N_t(z) - \Gamma_{f_t^\dD(z)}^{\I_2}N_t^\dD(z)}
\end{equation*}
where $\aA$ is the hyperbolic area of the triangle $\left(\I_2,f_t(z),f_t^\dD(z)\right)$. Using Heron's formula in $\Hh^2$ (see \cite{heron}, p.66), Proposition \ref{propConvUnitFrame} and the first point of Theorem \ref{theoremPerturbedDelaunay},
\begin{equation*}
	\aA \leq  \distH{f_t(z)}{f_t^\dD(z)}\times \distH{\I_2}{f_t^\dD(z)} \leq C|t||z|^{1-\delta}.
\end{equation*}
Moreover, denoting by $Q_t$ and $Q_t^\dD$ the unitary parts of $F_t(e^q)$ and $F_t^\dD(e^q)$ in their polar decomposition and using Lemma \ref{lemmaPolarSymNor} together with Corollary \ref{corQ2v-Q1v} and Equation \eqref{eqFtilde-I2},
\begin{align*}
	\norm{\Gamma_{f_t(z)}^{\I_2}N_t(z) - \Gamma_{f_t^\dD(z)}^{\I_2}N_t^\dD(z)} &= \norm{Q_t(z)\cdot \sigma_3 - Q_t^\dD(z)\cdot \sigma_3}_{T_{\I_2}\Hh^3}\\
	&\leq C \norm{F_t^\dD(z)}_{{R}}^2\norm{F_t^\dD(z)^{-1}F_t^\dD(z)-\I_2}_{{R}}\\
	&\leq C |t||z|^{1-3\delta}.
\end{align*}

\subsection{Embeddedness}\label{sectionEmbeddedness}

In this section, we prove the second point of Theorem \ref{theoremPerturbedDelaunay}. We thus assume that $t>0$. We suppose that $C,\epsilon,T,\delta>0$ satisfy Proposition \ref{propConvNormales} and the first point of Theorem \ref{theoremPerturbedDelaunay}.

\begin{lemma}\label{lemmaDelaunayPointsEloignes}
	Let $r_t>0$ such that the tubular neighbourhood of $f_t^\dD(\Cc^*)$ with hyperbolic radius $r_t$ is embedded. There exists $T>0$ and  $0<\epsilon'<\epsilon$ such that for all $0<t<T$, $x\in\partial D_\epsilon$ and $y\in D_{\epsilon'}^*$,
	\begin{equation*}
	\distH{f_t^\dD(x)}{f_t^\dD(y)}>4r_t.
	\end{equation*}
\end{lemma}
\begin{proof}
	The convergence of $f_t^\dD(\Cc^*)$ towards a chain of spheres implies that $r_t$ tends to $0$ as $t$ tends to $0$. If $f_t^\dD$ does not degenerate into a point, then it converges towards the parametrisation of a sphere, and for all $0<\epsilon'<\epsilon$ there exists $T>0$ satisfying the inequality. If $f_t^\dD$ does degenerate into a point, then a suitable blow-up makes it converge towards a catenoidal immersion on the punctured disk $D_{\epsilon}^*$ (see Section \ref{sectionBlowUp}). This implies that for $\epsilon'>0$ small enough, there exists $T>0$ satisfying the inequality. 
\end{proof}

We can now prove embeddedness with the same method than in \cite{raujouan}. Let $\dD_t:=f_t^\dD\left(\Cc^*\right)\subset \Hh^3$ be the image Delaunay surface of $f_t^\dD$. We denote by $\eta_t^\dD:\dD_t\longrightarrow T\Hh^3$ the Gauss map of $\dD_t$. We also write $\mM_t=f_t(D_{\epsilon}^*)$ and $\eta_t:\mM_t\longrightarrow T\Hh^3$. Let $r_t$ be the maximal value of $\alpha$ such that the following map is a diffeomorphism:
\begin{equation*}
	\function{\tT}{\left(-\alpha,\alpha\right)\times \dD_t}{\Tub_\alpha\dD_t\subset\Hh^3}{(s,p)}{\geod\left(p,\eta_t^\dD(p)\right)(s).}
\end{equation*}
According to Lemma \ref{lemmaTubularRadius}, the maximal tubular radius satisfies $r_t\sim t$ as $t$ tends to $0$. Using the first point of Theorem \ref{theoremPerturbedDelaunay}, we thus assume that on $D_{\epsilon}^*$,
\begin{equation*}
	\distH{f_t(z)}{f_t^\dD(z)}\leq \alpha r_t
\end{equation*}
where $\alpha<1$ is given by Lemma \ref{lemmaCompNormalesDelaunay} of Appendix \ref{appendixJleliLopez}.

Let $\pi_t$ be the projection from $\Tub_{r_t}\dD_t$ to $\dD_t$. Then the map 
\begin{equation*}
	\function{\varphi_t}{D_\epsilon^*}{\dD_t}{z}{\pi_t\circ f_t(z)}
\end{equation*}
is well-defined and satisfies
\begin{equation}\label{eqDistPhitFtD}
	\distH{\varphi_t(z)}{f_t^\dD(z)}\leq 2\alpha r_t
\end{equation}
because of the triangular inequality.

\begin{lemma}
	For $t>0$ small enough, $\varphi_t$ is a local diffeomorphism on $D_{\epsilon}^*$.
\end{lemma}
\begin{proof}
	It suffices to show that for all $z\in D_{\epsilon}^*$,
	\begin{equation}\label{eqCompNormalesDiffeo}
	\norm{\Gamma_{{\varphi}_t(z)}^{f_t(z)}\eta_t^\dD(\varphi_t(z))-N_t(z)}<1.
	\end{equation}
	Using Lemma \ref{lemmaTriangIneq} (we drop the variable $z$ to ease the notation),
	\begin{equation*}
	\norm{\Gamma_{{\varphi}_t}^{f_t}\eta_t^\dD(\varphi_t)-N_t}\leq \aA + \norm{\Gamma_{{\varphi}_t}^{f_t^\dD}\eta_t^\dD(\varphi_t)-\Gamma_{f_t}^{f_t^\dD}N_t}
	\end{equation*}
	where $\aA$ is the area of the triangle $\left( f_t,f_t^\dD,\varphi_t \right)$. Recall the isoperimetric inequality in $\Hh^2$ (see \cite{teufel1991}):
	\begin{equation*}
	\pP^2\geq 4\pi\aA + \aA^2
	\end{equation*}
	from which we deduce
	\begin{equation*}\label{eqIsoperimetric}
	\aA \leq \pP^2 \leq \left( 2\distH{f_t}{f_t^\dD} + 2\distH{\varphi_t}{f_t^\dD} \right)^2\leq \left(6\alpha r_t\right)^2
	\end{equation*}
	which uniformly tends to $0$ as $t$ tends to $0$. Using the triangular inequality and Proposition \ref{propConvNormales},
	\begin{equation*}
	\norm{\Gamma_{{\varphi}_t}^{f_t^\dD}\eta_t^\dD(\varphi_t)-\Gamma_{f_t}^{f_t^\dD}N_t} \leq \norm{\Gamma_{{\varphi}_t}^{f_t^\dD}\eta_t^\dD(\varphi_t)-N_t^\dD} + C|t||z|^{1-\delta}
	\end{equation*}
	and the second term of the right-hand side uniformly tends to $0$ as $t$ tends to $0$. Because $\alpha$ satisfies Lemma \ref{lemmaCompNormalesDelaunay} in Appendix \ref{appendixJleliLopez},
	\begin{equation*}
	\norm{\Gamma_{{\varphi}_t}^{f_t^\dD}\eta_t^\dD(\varphi_t)-N_t^\dD} < 1
	\end{equation*}
	which implies Equation \eqref{eqCompNormalesDiffeo}.
\end{proof}

Let $\epsilon'>0$ given by Lemma \ref{lemmaDelaunayPointsEloignes}. The restriction
\begin{equation*}
\function{\widetilde{\varphi}_t}{\varphi_t^{-1}\left(\varphi_t(D_{\epsilon'}^*)\right)\cap D_\epsilon^*}{\varphi_t\left( D_{\epsilon'}^* \right)}{z}{\varphi_t(z)}
\end{equation*}
is a covering map because it is a proper local diffeomorphism between locally compact spaces. To show this, proceed by contradiction as in $\Rr^3$ (see \cite{raujouan}): let $(x_i)_{i\in\Nn}\subset \varphi_t^{-1}\left(\varphi_t(D_{\epsilon'}^*)\right)\cap D_\epsilon^*$ such that $\left(\widetilde{\varphi}_t(x_i)\right)_{i\in\Nn}$ converges to $p\in\varphi_t\left(D_{\epsilon'}^*\right)$. Then $(x_i)_i$ converges to $x\in\overline{D}_\epsilon$. Using Equation \eqref{eqDistPhitFtD} and the fact that $f_t^\dD$ has an end at $0$, $x\neq 0$. If $x\in\partial D_\epsilon$, denoting $\widetilde{x}\in D_{\epsilon'}^*$ such that $\widetilde{\varphi}_t(\widetilde{x}) = p$, one has
\begin{equation*}
\distH{f_t^\dD(x) }{ f_t^\dD(\widetilde{x})} < \distH{f_t^\dD(x) }{ p} + \distH{f_t^\dD(\widetilde{x})}{\widetilde{\varphi}_t(\widetilde{x})} < 4\alpha r_t < 4r_t
\end{equation*}
which contradicts the definition of $\epsilon'$.

Let us now prove as in \cite{raujouan} that $\widetilde{\varphi}_t$ is a one-sheeted covering map. Let $\gamma : [0,1] \longrightarrow D_{\epsilon'}^*$ be a loop of winding number $1$ around $0$, $\Gamma = f_t^\dD(\gamma)$ and $\widetilde{\Gamma} = \widetilde{\varphi}_t(\gamma) \subset\dD_t$ and let us construct a homotopy between $\Gamma$ and $\widetilde{\Gamma}$. For all $s\in\left[0,1\right]$, let $\sigma_s: \left[0,1\right]\longrightarrow \Hh^3$ be a geodesic arc joining $\sigma_s(0)=\widetilde{\Gamma}(s)$ to $\sigma_s(1) = \Gamma(s)$.
For all $s,r\in[0,1]$, $\distH{\sigma_s(r) }{ \Gamma(s)} \leq \alpha r_t$
which implies that $\sigma_s(r) \in \Tub_{r_t}\dD_t$ because $\dD_t$ is complete. One can thus define the following homotopy between $\Gamma$ and $\widetilde{\Gamma}$
\begin{equation*}
\function{H}{[0,1]^2}{\dD_t}{(r,s)}{\pi_t\circ \sigma_s(r)}
\end{equation*}
where $\pi_t$ is the projection from $\Tub_{r_t}\dD_t$ to $\dD_t$. Using the fact that $f_t^\dD$ is an embedding, the degree of $\Gamma$ is one, and the degree of $\widetilde{\Gamma}$ is also one. Hence, $\widetilde{\varphi}_t$ is one-sheeted.

Finally, the map $\widetilde{\varphi}_t$ is a one-sheeted covering map and hence a diffeomorphism, so $f_t\left(D_{\epsilon'}^*\right)$ is a normal graph over $\dD_t$ contained in its embedded tubular neighbourhood and $f_t\left(D_{\epsilon'}^*\right)$ is thus embedded, which proves the second point of Theorem \ref{theoremPerturbedDelaunay}.

\subsection{Limit axis}\label{sectionLimitAxis}

In this section, we prove the third point of Theorem \ref{theoremPerturbedDelaunay} and compute the limit axis of $f_t^\dD$ as $t$ tends to $0$. Recall that $f_t^\dD = \Sym_q\left(\Uni\left(M_tz^{A_t}\right)\right)$ where $M_t$ tends to $\I_2$ as $t$ tends to $0$. Hence, the limit axis of $f_t^\dD$ and $\widetilde{f}_t^\dD := \Sym_q\left(\Uni\left(z^{A_t}\right)\right)$ are the same. Two cases can occur, wether $r> s$ or $r< s$.

\paragraph{Spherical family.}

At $t=0$, $r=\frac{1}{2}$ and $s=0$. The limit potential is thus
\begin{equation*}
\xi_0(z,\la) = \begin{pmatrix}
0 & \frac{\la^{-1}}{2} \\
\frac{\la}{2} & 0
\end{pmatrix}z^{-1} dz.
\end{equation*}
Consider the gauge
\begin{equation*}
G(z,\la) = \frac{1}{\sqrt{2z}}\begin{pmatrix}
1 & 0\\
\la & 2z
\end{pmatrix}.
\end{equation*}
The gauge potential is then
\begin{equation*}
\xi_0\cdot G (z,\la) = \begin{pmatrix}
0 & \la^{-1}dz\\
0&0
\end{pmatrix} = \xi_\Ss(z,\la)
\end{equation*}
where $\xi_\Ss$ is the spherical potential as in Section \ref{sectionDPWMethod}. Let $\widetilde{\Phi}:=z^{A_0}G$ be the gauged holomorphic frame and compute
\begin{align*}
\widetilde{\Phi}(1,\la) &=  G(1,\la) \\
&= \frac{1}{\sqrt{2}}\begin{pmatrix}
1 & 0\\
\la & 2
\end{pmatrix} = \frac{1}{\sqrt{2}}\begin{pmatrix}
1 & -\la^{-1} \\
\la & 1
\end{pmatrix}\begin{pmatrix}
1 & \la^{-1} \\
0 & 1
\end{pmatrix}\\
&= H(\la)\Phi_\Ss(1,\la)
\end{align*}
where $\Phi_\Ss$ is defined in \eqref{eqPhi_Ss} and $H=H_0$ as in \eqref{eqH0spherical}. This means that $\widetilde{\Phi} = H\Phi_\Ss$, $\Uni \widetilde{\Phi} = H F_\Ss$ and $\Sym_q(\Uni \widetilde{\Phi}) = H(e^{-q})\cdot f_\Ss$ because $H\in\LSUdeuxR$. Thus, using Equations \eqref{eqftilde_Ss=gammav_Ss} and \eqref{eqVecteurvS},
\begin{align*}
\widetilde{f}_0^\dD(\infty) &= \Sym_q(\Uni \widetilde{\Phi})(\infty) = H(e^{-q})\cdot f_\Ss(\infty) \\
&= \left(H(e^{-q})R(q)\right)\cdot \geod\left( \I_2,\sigma_3 \right)(q)\\
&= H(e^{-q})\cdot \geod\left( \I_2,\sigma_3 \right)(2q).
\end{align*}
And with the same method,
\begin{equation*}
\widetilde{f}_0^\dD(0) = H(e^{-q})\cdot \geod\left(\I_2,\sigma_3\right)(0).
\end{equation*}
This means that the limit axis of $\widetilde{f}_t^\dD$ as $t\to 0$, oriented from $z=\infty$ to $z=0$ is given in the spherical family by
\begin{equation*}
H(e^{-q})\cdot {\geod}\left(\I_2,-\sigma_3\right).
\end{equation*}

\paragraph{Catenoidal family.}

We cannot use the same method as above, as the immersion $\widetilde{f}_t^\dD$ degenerates into the point $\I_2$. Use Proposition \ref{propBlowupCatenoid} of Section \ref{sectionBlowUp} to get
\begin{equation*}
\widehat{f}:=\lim\limits_{t\to 0} \frac{1}{t} \left(f_t - \I_2\right) =  {\psi} \subset T_{\I_2}\Hh^3
\end{equation*}
where ${\psi}$ is the immersion of a catenoid of axis oriented by $-\sigma_1$ as $z\to 0$. This suffices to show that the limit axis oriented from the end at $\infty$ to the end at $0$ of the catenoidal family $\widetilde{f}_t^\dD$ converges as $t$ tends to $0$ to the oriented geodesic $ {\geod}(\I_2,-\sigma_1) $.

\section{Gluing Delaunay ends to hyperbolic spheres}\label{sectionNnoids}

In this section, we follow step by step the method Martin Traizet used in $\Rr^3$ (\cite{nnoids}) to construct CMC $H>1$ $n$-noids in $\Hh^3$ and prove Theorem \ref{theoremConstructionNnoids}. This method relies on the Implicit Function Theorem and aim to find a couple $(\xi_t,\Phi_t)$ satisfying the hypotheses of Theorem \ref{theoremPerturbedDelaunay} around each pole of an $n$-punctured sphere. More precisely, Implicit Function Theorem is used to solve the monodromy problem around each pole and to ensure that the potential is regular at $z=\infty$. The set of equations characterising this problem at $t=0$ is the same than in \cite{nnoids}, and the partial derivative with respect to the parameters is the same than in \cite{nnoids} at $t=0$. Therefore, the Implicit Function Theorem can be used exactly as in \cite{nnoids} and we do net repeat it here. Showing that the surface has Delaunay ends involves slightly different computations, but the method is the same than in \cite{nnoids}, namely, find a suitable gauge and change of coordinates around each pole of the potential in order to retrieve a perturbed Delaunay potential as in Definition \ref{defPerturbedDelaunayPotential}. One can then apply Theorem \ref{theoremPerturbedDelaunay}.  Finally, we show that the surface is Alexandrov-embedded (and embedded in some cases) by adapting the arguments of \cite{minoids} to the case of $\Hh^3$.

\subsection{The DPW data}

Let $H>1$, $q=\arcoth H$ and $\rho>e^{q}$. Let $n\geq 3$ and $u_1,\cdots,u_n$ unitary vectors of $T_{\I_2}\Hh^3$. Suppose, by applying a rotation, that $u_i\neq \pm \sigma_3$ for all $i\in\left[1,n\right]$. Let $v_\Ss:  \Cc\cup\left\{ \infty \right\} \longrightarrow \Ss^2$ defined as in Equation \eqref{eqVecteurvS} and  $\pi_i := v_\Ss^{-1}(u_i)\in\Cc^*$. Consider $3n$ parameters $a_i$, $b_i$, $p_i \in \Lambda\Cc_\rho^{\geq 0}$ assembled into a vector $\varx$ which stands in a neighbourhood of a central value $\varx_0$ so that the central values of $a_i$ and $p_i$ are $\tau_i$ and $\pi_i$. Introduce a real parameter $t$ in a neighbourhood of $0$ and define
\begin{equation*}\label{defBetatNnoids}
\beta_t(\la) := t\left( \la - e^q \right)\left(\la - e^{-q}\right).
\end{equation*}
The potential we use is
\begin{equation*}\label{defpotentialnnoids}
	\xi_{t,\varx}(z,\la) := \begin{pmatrix}
	0 & \la^{-1}dz\\
	\beta_t(\la)\omega_\varx (z,\la) & 0
	\end{pmatrix}
\end{equation*}
where
\begin{equation*}\label{defomegannoids}
\omega_\varx (z,\la) := \sum_{i=1}^{n}\left( \frac{a_i(\la)}{(z-p_i(\la))^2} + \frac{b_i(\la)}{z-p_i(\la)} \right)dz.
\end{equation*}
The initial condition is the identity matrix, taken at the point $z_0=0 \in \Omega$ where
\begin{equation*}
	\Omega = \left\{ z\in\Cc \mid \forall i\in\left[ 1,n \right], |z-\pi_i|>\epsilon \right\} 
\end{equation*}
and $\epsilon>0$ is a fixed constant such that the disks $D(\pi_i,2\epsilon)\subset\Cc$ are disjoint and do not contain $0$. Although the poles $p_1,\dots,p_n$ of the potential $\xi_{t,\varx}$ are functions of $\la$, $\xi_{t,\varx}$ is well-defined on $\Omega$ for $\varx$ sufficiently close to $\varx_0$. We thus define $\Phi_{t,\varx}$ as the solution to the Cauchy problem \eqref{eqCauchyProblem} with data $(\Omega,\xi_{t,\varx},0,\I_2)$.

The main properties of this potential are the same than in \cite{nnoids}, namely: it is a perturbation of the spherical potential $\xi_{0,\varx}$ and the factor $\left(\la - e^{-q}\right)$ in $\beta_t$ ensures that the second equation of the monodromy problem \eqref{eqMonodromyProblem} is solved. 

Let $\{\gamma_1,\cdots,\gamma_{n-1}\}$ be a set of generators of the fundamental group $\pi_1(\Omega,0)$ and define for all $i\in [1,n-1]$
\begin{equation*}
M_i(t,\varx) := \mM_{\gamma_i}(\Phi_{t,\varx}).
\end{equation*} 
Noting that
\begin{equation*}
\la\in\Ss^1 \implies \la^{-1}\left( \la - e^q \right)\left(\la - e^{-q}\right) = -2\left(\cosh q - \Re \la\right) \in\Rr,
\end{equation*}
the unitarity of the monodromy is equivalent to
\begin{equation*}
	\widetilde{M}_i(t,\varx)(\la) := \frac{\la}{\beta_t(\la)}\log M_i(t,\varx)(\la)\in \Lsurho.
\end{equation*}
Note that at $t=0$, the expression above takes the same value than in \cite{nnoids}, and so does the regularity conditions. One can thus apply 
Propositions 2 and 3 of \cite{nnoids} which we recall in Proposition \ref{propPropsnnoids}  below.
\begin{proposition}\label{propPropsnnoids}
	For t in a neighbourhood of $0$, there exists a unique smooth map $t\mapsto \varx(t) = (a_{i,t},b_{i,t},p_{i,t})_{1\leq i\leq n}\in(\wWRgeqzero)^3$ such that $\varx(0) = \varx_0$, the monodromy problem and the regularity problem are solved at $(t,\varx(t))$ and the following normalisations hold:
	\begin{equation*}
		\forall i\in[1,n-1],\qquad  \Re (a_{i,t})\mid_{\la=0} = \tau_i \quad \text{and} \quad p_{i,t}\mid_{\la=0} = \pi_i.
	\end{equation*}
	Moreover, at $t=0$, $\varx_0$ is a constant with $a_i$ real and such that
	\begin{equation*}
		b_i = \frac{-2a_i\overline{p_i}}{1+|p_i|^2}\quad \text{and} \quad \sum_{i=1}^{n}a_iv_\Ss(p_i)=0.
	\end{equation*}
\end{proposition}

Now write $\omega_t := \omega_{\varx(t)}$,  $\xi_t := \xi_{t,\varx(t)}$ and apply the DPW method to define the holomorphic frame $\Phi_t$ associated to $\xi_t$ on the universal cover $\widetilde{\Omega}$ of $\Omega$ with initial condition $\Phi_t(0) = \I_2$. Let $F_t:=\Uni \Phi_t$ and $f_t:=\Sym_q F_t$. The monodromy problem for $\Phi_t$ being solved, $f_t$ descends to a well-defined CMC $H$ immersion on $\Omega$. Use Theorem 3 and Corollary 1 of \cite{nnoids} to extend $f_t$ to $\Sigma_t := \Cc\cup\left\{\infty\right\}\backslash\left\{p_{1,t}(0),\dots ,p_{n,t}(0) \right\}$ and define $M_t=f_t(\Sigma_t)$. Moreover, with the same proof as in \cite{nnoids} (Proposition 4, point (2)), $a_{i,t}$ is a real constant with respect to $\la$ for all $i$ and $t$.

\subsection{Delaunay ends}\label{sectionDelaunayEndsNnoids}

\paragraph{Perturbed Delaunay potential.}

Let $i\in\left[1,n\right]$. We are going to gauge $\xi_t$ around its pole $p_{i,t}(0)$ and show that the gauged potential is a perturbed Delaunay potential as in Definition \ref{defPerturbedDelaunayPotential}. Let $(r,s):(-T,T)\longrightarrow \Rr^2$ be the continuous solution to (see Section \ref{sectionDelaunaySurfaces} for details)
\begin{equation*}
\left\{
\begin{array}{l}
rs = ta_{i,t},\\
r^2+s^2+2rs \cosh q = \frac{1}{4},\\
r> s.
\end{array}
\right.
\end{equation*}
For all $t$ and $\la$, define $\psi_{i,t,\la}(z):=z+p_{i,t}(\la)$ and
\begin{equation*}
	G_t(z,\la) := \begin{pmatrix}
	\frac{\sqrt{z}}{\sqrt{r+s\la}} & 0 \\
	\frac{-\la}{2\sqrt{z}\sqrt{r+s\la}} & \frac{\sqrt{r+s\la}}{\sqrt{z}}
	\end{pmatrix}.
\end{equation*}
For $T$ small enough, one can thus define on a uniform neighbourhood of $0$ the potential
\begin{equation*}
	\widetilde{\xi}_{i,t}(z,\la) := ((\psi_{i,t,\la}^*\xi_t)\cdot G_t)(z,\la) =  \begin{pmatrix}
	0 & r\la^{-1}+s\\
	\frac{\beta_t(\la)}{r+s\la}(\psi_{i,t,\la}^*\omega_t(z))z^2+\frac{\la}{4(r+s\la)} & 0
	\end{pmatrix}z^{-1}dz.
\end{equation*}
Note that by definition of $r$, $s$ and $\beta_t$,
\begin{equation*}
\left(r+s\la\right)\left( r\la+s \right) = \frac{\la}{4} + \beta_t(\la)a_{i,t}
\end{equation*}
and thus
\begin{equation*}
\widetilde{\xi}_{i,t}(z,\la) = A_t(\la)z^{-1}dz + C_t(z,\la)dz
\end{equation*}
with $A_t$ as in Equation \eqref{eqArs} satisfies Equation \eqref{eqrst} and $C_t$ as in Definition \ref{defPerturbedDelaunayPotential}. The potential $\widetilde{\xi}_{i,t}$ is thus a perturbed Delaunay potential as in Definition \ref{defPerturbedDelaunayPotential}. Moreover, using Theorem 3 of \cite{nnoids}, the induced immersion $\widetilde{f}_{i,t}$ satisfies 
\begin{equation*}
	\widetilde{f}_{i,t} = \psi_{i,t,0}^*f_t.
\end{equation*}

\paragraph{Applying Theorem \ref{theoremPerturbedDelaunay}.}

The holomorphic frame $\widetilde{\Phi}_{i,t}:=\Phi_tG_{i,t}$ associated to $\widetilde{\xi}_{i,t}$ satisfies the regularity and monodromy hypotheses of Theorem \ref{theoremPerturbedDelaunay}, but at $t=0$ and $z=1$,
\begin{equation*}
	\widetilde{\Phi}_{i,0}(1,\la) = \begin{pmatrix}
	1 & \left(1+\pi_i\right)\la^{-1}\\
	0 & 1
	\end{pmatrix}\begin{pmatrix}
	\sqrt{2} & 0\\
	\frac{-\la}{\sqrt{2}} & \frac{1}{\sqrt{2}}
	\end{pmatrix} = \frac{1}{\sqrt{2}}\begin{pmatrix}
	1-\pi_i & (1+\pi_i)\la^{-1}\\
	-\la & 1
	\end{pmatrix} =: M_i(\la),
\end{equation*}
and thus $\widetilde{\Phi}_{i,0}(z) = M_iz^{A_0}$. Recall \eqref{eqHtDt} and let
\begin{equation*}\label{eqHnnoids}
H :=H_0 = \frac{1}{\sqrt{2}}\begin{pmatrix}
1 & -\la^{-1}\\
\la & 1
\end{pmatrix}\in\LSUdeuxrho
\end{equation*}
and $Q_i := \Uni\left( M_iH\right)$.
Using Lemma 1 in \cite{raujouan}, $Q_i$ can be made explicit and one can find a change of coordinates $h$ and a gauge $G$ such that $\widehat{\Phi}_{i,t} := HQ_i^{-1}(h^* \widetilde{\Phi}_{i,t})G$ solves  $d\widehat{\Phi}_{i,t} = \widehat{\Phi}_{i,t}\widehat{\xi}_{i,t}$ where $\widehat{\xi}_{i,t}$ is a perturbed Delaunay potential and $\widehat{\Phi}_{i,0}(z) = z^{A_0}$. Explicitely,
\begin{equation*}
	Q_i(\la) = \frac{1}{\sqrt{1+|\pi_i|^2}}\begin{pmatrix}
	1 & \la^{-1}\pi_i\\
	-\la\conj{\pi}_i & 1
	\end{pmatrix}
\end{equation*}
and
\begin{equation*}
	h(z) = \frac{(1+|\pi_i|^2)z}{1-\conj{\pi}_iz},\qquad G(z,\la) = \frac{1}{\sqrt{1-\conj{\pi}_iz}}\begin{pmatrix}
	1 & 0\\
	-\la\conj{\pi}_iz & 1-\conj{\pi}_iz
	\end{pmatrix}.
\end{equation*}
One can thus apply Theorem \ref{theoremPerturbedDelaunay} on $\widehat{\xi}_{i,t}$ and $\widehat{\Phi}_{i,t}$, which proves the existence of the family $\left(M_t\right)_{0<t<T}$ of CMC $H$ surfaces of genus zero and $n$ Delaunay ends, each of weight (according to Equation \eqref{eqPoidsDelaunayPure})
\begin{equation*}
	w_{i,t} = 8\pi  rs\sinh q = \frac{8\pi t a_{i,t}}{\sqrt{H^2-1}},
\end{equation*}
which proves the first point of Theorem \ref{theoremConstructionNnoids} (after a normalisation on $t$). Let $\widehat{f}_{i,t}:=\Sym_q(\Uni \widehat{\Phi}_{i,t})$ and $\widehat{f}_{i,t}^\dD$ the Delaunay immersion given by Theorem \ref{theoremPerturbedDelaunay}.

\paragraph{Limit axis.}

In order to compute the limit axis of $f_t$ at the end around $p_{i,t}$, let $\widehat{\Delta}_{i,t}$ be the oriented axis of $\widehat{f}_{i,t}^\dD$ at $w=0$. Then, using Theorem \ref{theoremPerturbedDelaunay},
\begin{equation*}
	\widehat{\Delta}_{i,0} = H(e^{-q})\cdot \geod\left( \I_2, -\sigma_3 \right).
\end{equation*}
And using $\widehat{f}_{i,t}(w) = H(e^{-q})Q_i(e^{-q})^{-1}\cdot \left( h^*f_t(z) \right)$,
\begin{equation*}
	\widehat{\Delta}_{i,0} = H(e^{-q})Q_i(e^{-q})^{-1}\cdot\Delta_{i,0}
\end{equation*}
and thus
\begin{equation*}
	\Delta_{i,0} = Q_i(e^{-q}) \cdot \geod(\I_2,-\sigma_3).
\end{equation*}
Computing $M_iH = \Phi_\Ss(\pi_i)$ as in \eqref{eqPhi_Ss}, one has $Q_i=F_\Ss(\pi_i)$. Hence
\begin{equation*}
\Delta_{i,0} = {\geod}\left( f_\Ss(\pi_i), -N_\Ss(\pi_i) \right)
\end{equation*}
where $N_\Ss$ is the normal map associated to $\Phi_\Ss$. Using Equation \eqref{eqR(q)}, $f_\Ss(z) = R(q)\cdot \widetilde{f}_\Ss(z)$ and $N_\Ss(z) =R(q) \cdot \widetilde{N}_\Ss(z)$ where $\widetilde{N}_\Ss$ is the normal map of $\widetilde{f}_\Ss$. Using Equation \eqref{eqftilde_Ss=gammav_Ss} and the fact that $\widetilde{f}_\Ss$ is a spherical immersion gives
\begin{equation*}
\widetilde{N}_\Ss(z) = \Gamma_{\I_2}^{\widetilde{f}_\Ss(z)}\left(-v_\Ss(z)\right)
\end{equation*}
and thus
\begin{align*}
\Delta_{i,0} &= {\geod}\left( R(q)\cdot \widetilde{f}_\Ss(\pi_i), -R(q)\cdot \widetilde{N}_\Ss(\pi_i) \right)\\
&= R(q)\cdot {\geod}\left(\widetilde{f}_\Ss(\pi_i),\Gamma_{\I_2}^{\widetilde{f}_\Ss(\pi_i)}v_\Ss(\pi_i)\right)\\
&= R(q)\cdot {\geod}\left(\I_2,u_i\right).
\end{align*}
Apply the isometry given by $R(q)^{-1}$ and note that $R(q)$ does not depend on $i$ to prove point 2 of Theorem \ref{theoremConstructionNnoids}.

\subsection{Embeddedness}

We suppose that $t>0$ and that all the weights $\tau_i$ are positive, so that the ends of $f_t$ are embedded. Recall the definition of Alexandrov-embeddedness (as stated in \cite{minoids}):

\begin{definition}\label{defAlexandrovEmbedded}
	A surface $M^2\subset\mM^3$ of finite topology is Alexandrov-embedded if $M$ is properly immersed, if each end of $M$ is embedded, and if there exists a compact $3$-manifold $\overline{W}$ with boundary $\partial \overline{W} = \overline{S}$, $n$ points $p_1,\cdots,p_n\in\overline{S}$ and a proper immersion $F: W=\overline{W}\backslash\{ p_1,\cdots,p_n \}\longrightarrow \mM$ whose restriction to $S=\overline{S}\backslash\{ p_1,\cdots, p_n \}$ parametrises $M$.
\end{definition}

The following lemma is proved in \cite{minoids} in $\Rr^3$ and for surfaces with catenoidal ends, but the proof is the same in $\Hh^3$ for surfaces with any type of embedded ends. For any oriented surface $M$ with Gauss map $N$ and any $r>0$, the tubular map of $M$ with radius $r$ is defined by
\begin{equation*}
	\function{\tT}{(-r,r)\times M}{\Tub_{r}M}{(s,p)}{\geod(p,N(p))(s).}
\end{equation*}

\begin{lemma}\label{lemmaExtendingAlexandrov}
	Let $M$ be an oriented Alexandrov-embedded surface of $\Hh^3$ with $n$ embedded ends. Let $r>0$ and suppose that the tubular map of $M$ with radius $r$ is a local diffeomorphism. With the notations of Definition \ref{defAlexandrovEmbedded}, there exist a hyperbolic $3$-manifold $W'$ containing $W$ and a local isometry $F': W'\longrightarrow  \Hh^3$ extending $F$ such that the tubular neighbourhood $\Tub_rS$  is embedded in $W'$.
\end{lemma}

In order to show that $M_t$ is embedded, we will use the techniques of \cite{minoids}. Thus, we begin by lifting $M_t$ to $\Rr^3$ with the exponential map at the identity $\exp_{\I_2}: \Rr^3\longrightarrow \Hh^3$. This map is a diffeomorphism, so $M_t$ is Alexandrov-embedded if, and only if its lift $\widehat{M}_t$ to $\Rr^3$ given by the immersion
\begin{equation*}
	\widehat{f}_t := \exp_{\I_2}^{-1}\circ f_t: \Sigma_t\longrightarrow\Rr^3
\end{equation*}
is Alexandrov-embedded.

Let $T,\epsilon>0$ such that $f_t$ (and hence $\widehat{f}_t$) is an embedding of $D^*(p_{i,t},\epsilon)$ for all $i\in[1,n]$ and let $f_{i,t}^\dD: \Cc\backslash\{ p_{i,t} \}\longrightarrow\Hh^3$ be the Delaunay immersion approximating $f_t$ in $D^*(p_{i,t},\epsilon)$. Let $\widehat{f}_{i,t}^\dD:=\exp_{\I_2}^{-1}\circ f_{i,t}^\dD$. Apply an isometry of $\Hh^3$ so that the limit immersion $f_0$ maps $\Sigma_0$ to a $n$-punctured geodesic spere of hyperbolic radius $q$ centered at $\I_2$. Then $\widehat{f}_0(\Sigma_0)$ is a Euclidean sphere of radius $q$ centered at the origin. Define 
\begin{equation*}
	\function{\widehat{N}_t}{\Sigma_t}{\Ss^2}{z}{d(\exp_{\I_2}^{-1})(f_t(z))N_t(z).}
\end{equation*}
At $t=0$, $\widehat{N}_0$ is the normal map of $\widehat{f}_0$ (by Gauss Lemma), but not for $t>0$ because the Euclidean metric of $\Rr^3$ is not the metric induced by $\exp_{\I_2}$.

Let 
\begin{equation*}
	\function{h_i}{\Rr^3}{\Rr}{x}{\scal{x}{-\widehat{N}_0(p_{i,0})}}
\end{equation*}
be the height function in the direction of the limit axis.

As in \cite{minoids}, one can show that

\begin{claim}
	There exist $\delta<\delta'$ and $0<\epsilon'<\epsilon$ such that for all $i\in[1,n]$ and $0<t<T$,
	\begin{equation*}
		\max_{C(p_{i,t},\epsilon)} h_i\circ \widehat{f}_t < \delta < \min_{C(p_{i,t},\epsilon')}h_i\circ\widehat{f}_t \leq \max_{C(p_{i,t},\epsilon')}h_i\circ\widehat{f}_t < \delta'.
	\end{equation*} 
\end{claim}

Define for all $i$ and $t$:
\begin{equation*}
	\gamma_{i,t} := \left\{ z\in D^*_{p_{i,t},\epsilon} \mid h_i\circ\widehat{f}_t(z) = \delta \right\},\quad \gamma_{i,t}' := \left\{ z\in D^*_{p_{i,t},\epsilon'} \mid h_i\circ\widehat{f}_t(z) = \delta' \right\}.
\end{equation*}
From their convergence as $t$ tends to $0$,
\begin{claim}
	The regular curves $\gamma_{i,t}$ and $\gamma_{i,t}'$ are topological circles around $p_{i,t}$.
\end{claim}

Define $D_{i,t}, D_{i,t}'$ as the topological disks bounded by $\gamma_{i,t}, \gamma_{i,t}'$, and $\Delta_{i,t},\Delta_{i,t}'$ as the topological disks bounded by $\widehat{f}_t(\gamma_{i,t}), \widehat{f}_t(\gamma_{i,t}')$. Let $\aA_{i,t}:=D_{i,t}\backslash D_{i,t}'$. Then $\widehat{f}_t(\aA_{i,t})$ is a graph over the plane $\{h_i(x)=\delta\}$. Moreover, for all $z\in D_{i,t}^*$, $h_i\circ \widehat{f}_t(z)\geq \delta' > \delta$. Thus,
\begin{claim}
	The intersection $\widehat{f}_t(D_{i,t}^*)\cap \Delta_{i,t}$ is empty.
\end{claim}

	Define a sequence $(R_{i,t,k})$ such that $\widehat{f}_t(D_{i,t}^*)$ transversally intersects the planes $\{ h_i(x) = R_{i,t,k} \}$. Define
\begin{equation*}
\gamma_{i,t,k} := \left\{ z\in D_{i,t}^* \mid h_i\circ\widehat{f}_{i,t}(z) = R_{i,t,k} \right\},
\end{equation*}
and the topological disks $\Delta_{i,t,k}\subset \{ h_i(x) = R_{i,t,k} \} $ bounded by $\widehat{f}_t(\gamma_{i,t,k})$. Define $\aA_{i,t,k}$ as the annuli bounded by $\gamma_{i,t}$ and $\gamma_{i,t,k}$. Define ${W}_{i,t,k}\subset \Rr^3$ as the interior of $\widehat{f}_t(\aA_{i,t,k})\cup \Delta_{i,t}\cup \Delta_{i,t,k}$ and 
\begin{equation*}
{W}_{i,t} := \bigcup_{k\in \Nn} {W}_{i,t,k}.
\end{equation*}
Hence,
\begin{claim}
	The union $\widehat{f}_t(D_{i,t}^*)\cup \Delta_{i,t}$ is the boundary of a topological punctured ball ${W}_{i,t}\subset \Rr^3$.
\end{claim}

The union
\begin{equation*}
\widehat{f}_t\left( \Sigma_t\backslash \left( D_{1,t}\cup\cdots D_{n,t} \right) \right) \cup \Delta_{1,t}\cup \cdots \cup \Delta_{n,t}
\end{equation*}
is the boundary of a topological ball ${W}_{0,t}\subset \Rr^3$. Take
\begin{equation*}
W_t:= W_{0,t}\cup W_{1,t}\cup\cdots\cup W_{n,t}
\end{equation*}
to show that $\widehat{M}_t$, and hence $M_t$ is Alexandrov-embedded for $t>0$ small enough.

\begin{lemma}
	Let $S\subset\Hh^3$ be a sphere of hyperbolic radius $q$ centered at $p\in\Hh^3$. Let $n\geq 2$ and $\left\{ u_i \right\}_{i\in [1,n]}\subset T_p\Hh^3$. Let $\left\{ p_i \right\}_{i\in[1,n]}$ defined by $p_i=S\cap \geod(p,u_i)(\Rr_+)$.
	For all $i\in[1,n]$, let $S_i\subset\Hh^3$ be the sphere of hyperbolic radius $q$ such that $S\cap S_i =\{ p_i\}$. For all $(i,j)\in[1,n]^2$, let $\theta_{ij}$ be the angle between $u_i$ and $u_j$.
	
	If for all $i\neq j$, 
	\begin{equation*}
	\theta_{ij} > 2\arcsin\left(\frac{1}{2\cosh q}\right)
	\end{equation*}
	then $S_i\cap S_j = \emptyset$ for all $i\neq j$.
\end{lemma}
\begin{proof}
	Without loss of generality, we assume that $p=\I_2$. We use the ball model of $\Hh^3$ equipped with its metric
	\begin{equation*}
	ds_\Bb^2(x) = \frac{4ds_E^2}{\left(1-\norm{x}_E^2\right)^2}
	\end{equation*} 
	where $ds_E$ is the euclidean metric and $\norm{x}_E$ is the euclidean norm. In this model, the sphere $S$ is centered at the origin and has euclidean radius $R=\tanh \frac{q}{2}$. For all $i\in[1,n]$, the sphere $S_i$ has euclidean radius 
	\begin{equation*}
	r = \frac{1}{2}\left(\tanh\frac{3q}{2} - \tanh\frac{q}{2}\right) = \frac{\tanh \frac{q}{2}}{2\cosh q -1}.
	\end{equation*}
	Let $j\neq i$. In order to have $S_i\cap S_j = \emptyset$, one must solve
	\begin{equation*}
		(R+r)\sin\frac{\theta_{ij}}{2} \geq r,
	\end{equation*}
	which gives the expected result.
\end{proof}

In order to prove the last point of Theorem \ref{theoremConstructionNnoids}, just note that 
\begin{equation*}
	H = \coth q \implies \frac{1}{2\cosh q}=\frac{\sqrt{H^2-1}}{2H}. 
\end{equation*}
Suppose that the angle $\theta_{ij}$ between $u_i$ and $u_j$ satisfies Equation \eqref{eqAnglesNnoid} for all $i\neq j$. Then for $t>0$  small enough, the proper immersion $F_t$ given by Definition \ref{defAlexandrovEmbedded} is injective (because of the convergence towards a chain of spheres) and hence $M_t$ is embedded.

\begin{remark}
	This means for example that in hyperbolic space, one can construct embedded CMC $n$-noids with seven coplanar ends or more.
\end{remark}

\section{Gluing Delaunay ends to minimal $n$-noids}\label{sectionMinoids}

Again, this section is an adaptation of Traizet's work in \cite{minoids} applied to the proof of Theorem \ref{theoremConstructionMinoids}. We first give in Section \ref{sectionBlowUp} a blow-up result for CMC $H>1$ surfaces in Hyperbolic space. We then introduce in Section \ref{sectionDataMinoids} the DPW data giving rise to the surface $M_t$ of Theorem \ref{theoremConstructionMinoids} and prove the convergence towards the minimal $n$-noid. Finally, using the same arguments than in \cite{minoids}, we prove Alexandrov-embeddedness in Section \ref{sectionAlexEmbeddedMnoids}.

\subsection{A blow-up result}\label{sectionBlowUp}

As in $\Rr^3$ (see \cite{minoids}), the DPW method accounts for the convergence of CMC $H>1$ surfaces in $\Hh^3$ towards minimal surfaces of $\Rr^3$ (after a suitable blow-up). We work with the following Weierstrass parametrisation:
\begin{equation}\label{eqWeierstrass}
W(z) = W(z_0) + \Re \int_{z_0}^z\left( \frac{1}{2}(1-g^2)\omega, \frac{i}{2}(1+g^2)\omega, g\omega \right)
\end{equation}

\begin{proposition}\label{propBlowUpDPW}
	Let $\Sigma$ be a Riemann surface, $(\xi_t)_{t\in I}$ a family of DPW potentials on $\Sigma$ and $(\Phi_t)_{t\in I}$ a family of solutions to $d\Phi_t = \Phi_t\xi_t$ on the universal cover $\widetilde{\Sigma}$ of $\Sigma$, where $I\subset \Rr$ is a neighbourhood of $0$. Fix a base point $z_0\in\widetilde{\Sigma}$ and $\rho>e^{q} >1$.  Assume that
	\begin{enumerate}
		\item $(t,z)\mapsto \xi_t(z)$ and $t\mapsto \Phi_t(z_0)$ are $\cC^1$ maps into $\Omega^1(\Sigma,\LsldeuxCrho)$ and $\LSLdeuxCrho$ respectively.
		\item For all $t\in I$, $\Phi_t$ solves the monodromy problem \eqref{eqMonodromyProblem}.
		\item $\Phi_0(z,\la)$ is independant of $\la$:
		\begin{equation*}
		\Phi_0(z,\la) = \begin{pmatrix}
		a(z) & b(z)\\
		c(z) & d(z) 
		\end{pmatrix}.
		\end{equation*}
	\end{enumerate}
	Let $f_t = \Sym_q\left(\Uni(\Phi_t)\right): \Sigma \longrightarrow \Hh^3$ be the CMC $H=\coth q$ immersion given by the DPW method. Then, identifying $T_{\I_2}\Hh^3$ with $\Rr^3$ via the basis $(\sigma_1,\sigma_2,\sigma_3)$ defined in \eqref{eqPauliMatrices},
	\begin{equation*}
	\lim\limits_{t\to 0}\frac{1}{t}\left(f_t - \I_2\right) =  W
	\end{equation*}
	where $W$ is a (possibly branched) minimal immersion with the following Weierstrass data:
	\begin{equation*}
	g(z) = \frac{a(z)}{c(z)}, \quad \omega(z) = -4(\sinh q) c(z)^2 \frac{\partial \xi_{t,12}^{(-1)}(z)}{\partial t}\mid_{t=0}.
	\end{equation*}
	The limit is for the uniform $\cC^1$ convergence on compact subsets of $\Sigma$.
\end{proposition}
\begin{proof}
	With the same arguments as in \cite{minoids}, $(t,z)\mapsto \Phi_t(z)$, $(t,z)\mapsto F_t(z)$ and $(t,z)\mapsto B_t(z)$ are $\cC^1$ maps into $\LSLdeuxCrho$, $\LSUdeuxrho$ and $\LplusRSLdeuxC_{\rho}$ respectively. At $t=0$, $\Phi_0$ is constant. Thus, $F_0$ and $B_0$ are constant with respect to $\la$:
	\begin{equation*}
	F_0 = \frac{1}{\sqrt{|a|^2+|c|^2}}\begin{pmatrix}
	a & -\bar{c}\\
	c & \bar{a}
	\end{pmatrix}, \quad B_0 = \frac{1}{\sqrt{|a|^2+|c|^2}} \begin{pmatrix}
	|a|^2+|c|^2 & \bar{a}b + \bar{c}d\\
	0 & 1
	\end{pmatrix}.
	\end{equation*}
	Thus, $F_0(z,e^{-q})\in\SU(2)$ and $f_0(z)$ degenerates into the identity matrix. Let $b_t := B_{t,11}\mid_{\la=0}$ and $\beta_t$ the upper-right residue at $\la=0$ of the potential $\xi_t$. Recalling Equation \eqref{eqdf},
	\begin{equation*}
	df_t(z) = 2b_t(z)^2\sinh q F_t(z,e^{-q})\begin{pmatrix}
	0 & \beta_t(z)\\
	\conj{\beta}_t(z) & 0
	\end{pmatrix}{F_t(z,e^{-q})}^*.
	\end{equation*}
	Hence $(t,z)\mapsto df_t(z)$ is a $\cC^1$ map. At $t=0$, $\xi_0 = \Phi_0^{-1}d\Phi_0$ is constant with respect to $\la$, so $\beta_0=0$ and $df_0(z) = 0$. Define $\widetilde{f}_t(z) := \frac{1}{t}\left( f_t(z) - \I_2 \right)$ for $t\neq 0$. Then $d\widetilde{f}_t(z)$ extends at $t=0$, as a continuous function of $(t,z)$ by
	\begin{align*}
	d\widetilde{f}_0 = \frac{d}{dt}df_t\mid_{t=0} &= 2\sinh q\begin{pmatrix}
	a & -\conj{c}\\
	c & \conj{a}
	\end{pmatrix}\begin{pmatrix}
	0 & \beta'\\
	\conj{\beta'} & 0
	\end{pmatrix}\begin{pmatrix}
	\conj{a} & \conj{c}\\
	-c & a
	\end{pmatrix}\\
	&= 2 \sinh q\begin{pmatrix}
	-ac\beta' - \conj{ac\beta'} & a^2\beta' - \conj{c^2\beta'}\\
	\conj{a^2\beta'} - c^2\beta' &  ac\beta' + \conj{ac\beta'}
	\end{pmatrix}
	\end{align*}
	where $\beta' = \frac{d}{dt}\beta_t\mid_{t=0}$. In $T_{\I_2}\Hh^3$, this gives
	\begin{equation*}
	d\widetilde{f}_0 = 4\sinh q\Re \left(  \frac{1}{2}\beta'(a^2-c^2), \frac{-i}{2}\beta'(a^2+c^2), -ac\beta' \right).
	\end{equation*}
	Writing $g=\frac{a}{c}$ and $\omega = -4c^2\beta' \sinh q$ gives:
	\begin{equation*}
	\widetilde{f}_0(z) = \widetilde{f}_0(z_0) + \Re  \int_{z_0}^{z} \left( \frac{1}{2}(1-g^2)\omega, \frac{i}{2}(1+g^2)\omega, g\omega \right).
	\end{equation*}
\end{proof}

As a useful example for Proposition \ref{propBlowUpDPW}, one can show the convergence of Delaunay surfaces in $\Hh^3$ towards a minimal catenoid.

\begin{proposition}\label{propBlowupCatenoid}
	Let $q>0$, $A_t = A_{r,s}$ as in \eqref{eqArs} with $r\leq s$ satisfying \eqref{eqrst}. Let $\Phi_t(z):=z^{A_t}$ and $f_t := \Sym_q\left(\Uni \Phi_t\right)$. Then 
	\begin{equation*}
	\widetilde{f} := \lim\limits_{t\to 0}\frac{1}{t}\left(f_t-\I_2\right) = \psi
	\end{equation*}
	where $\psi:\Cc^*\longrightarrow\Rr^3$ is the immersion of a catenoid centered at $(0,0,1)$, of neck radius $1$ and of axis orientd by the positive $x$-axis in the direction from $z=0$ to $z=\infty$.
\end{proposition}
\begin{proof}
	Compute
	\begin{equation*}
	\Phi_0(z,\la) = \begin{pmatrix}
	\cosh\left(\frac{\log z}{2}\right) & \sinh\left(\frac{\log z}{2}\right)\\
	\sinh\left(\frac{\log z}{2}\right) & \cosh\left(\frac{\log z}{2}\right)
	\end{pmatrix}
	\end{equation*}
	and 
	\begin{equation*}
	\frac{\partial \xi_{t,12}^{(-1)}(z)}{\partial t}\mid_{t=0} = \frac{z^{-1}dz}{2\sinh q}
	\end{equation*}
	in order to apply Proposition \ref{propBlowUpDPW} and get
	\begin{equation*}
	\widetilde{f}(z) = \widetilde{f}(1) + \Re  \int_{1}^{z} \left( \frac{1}{2}(1-g^2)\omega, \frac{i}{2}(1+g^2)\omega, g\omega \right)
	\end{equation*}
	where
	\begin{equation*}
	g(z) = \frac{z+1}{z-1}\qquad\text{and}\qquad \omega(z) = \frac{-1}{2}\left(\frac{z-1}{z}\right)^2dz.
	\end{equation*}
	Note that $\Phi_t(1) = \I_2$ for all $t$ to show that $\widetilde{f}(1) = 0$ and get
	\begin{equation*}
	\widetilde{f}(z) = \Re  \int_{1}^{z} \left( w^{-1}dw, \frac{-i}{2}(1+w^2)w^{-2}dw, \frac{1}{2}(1-w^2)w^{-2}dw \right).
	\end{equation*}
	Integrating gives for $(x,y)\in\Rr\times [0,2\pi]$:
	\begin{equation*}
	\widetilde{f}(e^{x+iy}) = \psi(x,y)
	\end{equation*}
	where
	\begin{equation*}
	\function{\psi}{\Rr\times [0,2\pi]}{\Rr^3}{(x,y)}{\left( x,\cosh(x)\sin(y), 1 - \cosh(x)\cos(y) \right)}
	\end{equation*}
	and hence the result.
\end{proof}

\subsection{The DPW data}\label{sectionDataMinoids}

In this Section, we introduce the DPW data inducing the surface $M_t$ of Theorem \ref{theoremConstructionMinoids}. The method is very similar to Section \ref{sectionNnoids} and to \cite{minoids}, which is why we omit the details.

\paragraph{The data.}

Let $(g,\omega)$ be the Weierstrass data (for the parametrisation defined in \eqref{eqWeierstrass}) of the minimal $n$-noid $M_0\subset \Rr^3$. 
If necessary, apply a M\"obius transformation so that $g(\infty)\notin\left\{0,\infty \right\}$, and write
\begin{equation*}
g(z) = \frac{A(z)}{B(z)}, \qquad \omega(z) = \frac{B(z)^2dz}{\prod_{i=1}^{n}(z-p_{i,0})^2}.
\end{equation*}
Let $H>1$, $q>0$ so that $H=\coth q$ and $\rho>e^{q}$. Consider $3n$ parameters $a_i,b_i,p_i\in\Lambda\Cc_\rho$ ($i\in[1,n]$) assembled into a vector $\varx$. Let 
\begin{equation*}
A_\varx(z,\la) = \sum_{i=1}^{n} a_i(\la)z^{n-1}, \qquad B_\varx(z,\la) = \sum_{i=1}^{n}b_i(\la)z^{n-1}
\end{equation*}
and
\begin{equation*}
g_\varx(z,\la) = \frac{A_\varx(z,\la)}{B_\varx(z,\la)}, \qquad \omega_\varx(z,\la) = \frac{B_\varx(z,\la)^2dz}{\prod_{i=1}^{n}(z-p_i(\la))^2}.
\end{equation*}
The vector $\varx$ is chosen in a neighbourhood of a central value $\varx_0\in\Cc^{3n}$ so that $A_{\varx_0}=A$, $B_{\varx_0} = B$ and $\omega_{\varx_0} = \omega$. Let $p_{i,0}$ denote the central value of $p_i$. Introduce a real parameter $t$ in a neighbourhood of $0$ and write
\begin{equation*}
\beta_t(\la) := \frac{t(\la-e^q)(\la-e^{-q})}{4\sinh q}.
\end{equation*}
The potential we use is
\begin{equation*}
\xi_{t,\varx}(z,\la) = \begin{pmatrix}
0 & \la^{-1}\beta_t(\la)\omega_\varx(z,\la)\\
d_zg_\varx(z,\la) & 0
\end{pmatrix}
\end{equation*}
defined for $(t,\varx)$ sufficiently close to $(0,\varx_0)$ on
\begin{equation*}
	\Omega = \left\{ z\in\Cc \mid \forall i\in\left[ 1,n \right], |z-p_{i,0}|>\epsilon \right\}\cup \left\{\infty \right\}
\end{equation*}
where $\epsilon>0$ is a fixed constant such that the disks $D(p_{i,0},2\epsilon)$ are disjoint. The initial condition is 
\begin{equation*}
	\phi(\la) = \begin{pmatrix}
	ig_\varx(z_0,\la) & i\\
	i & 0
	\end{pmatrix}
\end{equation*}
taken at $z_0\in\Omega$ away from the poles and zeros of $g$ and $\omega$. Let $\Phi_{t,\varx}$ be the holomorphic frame arising from the data $\left(\Omega,\xi_{t,\varx},z_0,\phi\right)$ via the DPW method and $f_{t,\varx}:=\Sym_q\left(\Uni \Phi_{t,\varx}\right)$.

Follow Section 6 of \cite{minoids} to show that the potential $\xi_{t,\varx}$ is regular at the zeros of $B_\varx$ and to solve the monodromy problem around the poles at $p_{i,0}$ for $i\in\left[1,n-1\right]$. The Implicit Function Theorem allows us to define $\varx = \varx(t)$ in a small neighbourhood $(-T,T)$ of $t=0$ satisfying $\varx(0) = \varx_0$ and such that the monodromy problem is solved for all $t$. We can thus drop from now on the index $\varx$ in our data. As in \cite{minoids}, $f_t$ descends to $\Omega$ and analytically extends to $\Cc\cup \{\infty\}\backslash\left\{ p_{1,0},\dots,p_{n,0} \right\}$. This defines a smooth family $(M_t)_{-T<t<T}$ of CMC $H$ surfaces of genus zero with $n$ ends in $\Hh^3$.

The convergence of $\frac{1}{t}\left( M_t-\I_2 \right)$ towards the minimal $n$-noid $M_0$ (point 2 of Theorem \ref{theoremConstructionMinoids}) is a straightforward application of Proposition \ref{propBlowUpDPW} together with
\begin{equation*}
\frac{\Phi_{0,11}(z)}{\Phi_{0,21}(z)} = g(z),\quad -4\left(\sinh q\right)\left(\Phi_{0,21}(z)\right)^2\frac{\partial\xi_{t,12}^{(-1)}(z)}{\partial t} = \omega(z).
\end{equation*}

\paragraph{Delaunay residue.}

To show that $\xi_t$ is a perturbed Delaunay potential around each of its poles, let $i\in[1,n]$ and follow Section \ref{sectionDelaunayEndsNnoids} with
\begin{equation*}
	\psi_{i,t,\la}(z) = g_t^{-1}\left( z+g_t(p_{i,t}(\la)) \right).
\end{equation*}
Define
\begin{equation*}
	\widetilde{\omega}_{i,t}(z,\la):=\psi_{i,t,\la}^*\omega_t(z)
\end{equation*}
and
\begin{equation*}
	\alpha_{i,t}(\la) := \Res_{z=0}(z\widetilde{\omega}_{i,t}(z,\la)).
\end{equation*}
Use Proposition 5, Claim 1 of \cite{minoids} to show that for $T$ small enough, $\alpha_{i,t}$ is real and does not depend on $\la$. Set
\begin{equation*}
\left\{
\begin{array}{l}
rs = \frac{t\alpha_{i,t}}{4\sinh q},\\
r^2+s^2+2rs\cosh q = \frac{1}{4},\\
r<s
\end{array}
\right.
\end{equation*}
and
\begin{equation*}
	G_{t}(z,\la) = \begin{pmatrix}
	\frac{\sqrt{r\la+s}}{\sqrt{z}} & \frac{-1}{2\sqrt{r\la+s}\sqrt{z}}\\
	0 & \frac{\sqrt{z}}{\sqrt{r\la+s}}
	\end{pmatrix}.
\end{equation*}
Define the gauged potential
\begin{equation*}
	\widetilde{\xi}_{i,t}(z,\la) := \left( (\psi_{i,t,\la}^*\xi_t)\cdot G_t \right)(z,\la)
\end{equation*}
and compute its residue to show that it is a perturbed Delaunay potential as in Definition \ref{defPerturbedDelaunayPotential}.

\paragraph{Applying Theorem \ref{theoremPerturbedDelaunay}.}

At $t=0$ and $z=1$, writing $\pi_i:=g(p_{i,0})$ to ease the notation,
\begin{equation*}
	\widetilde{\Phi}_{i,0}(1,\la) = \begin{pmatrix}
	i\left( 1+\pi_i \right) & i\\
	i&0
	\end{pmatrix}\begin{pmatrix}
	\frac{1}{\sqrt{2}} & \frac{-1}{\sqrt{2}} \\
	0 & \sqrt{2}
	\end{pmatrix} = \frac{i}{\sqrt{2}}\begin{pmatrix}
	1+\pi_i & 1-\pi_i\\
	1  & -1
	\end{pmatrix} =: M_i,
\end{equation*}
and thus, $\widetilde{\xi}_{i,0}(z) = M_iz^{A_0}$. Recall \eqref{eqHtDt} and let
\begin{equation*}
	H:= H_0 = \frac{1}{\sqrt{2}}\begin{pmatrix}
	1 & -1 \\
	1 & 1
	\end{pmatrix}\in\LSUdeuxrho
\end{equation*}
and $Q_i := \Uni\left( M_iH^{-1} \right)$. Using Lemma 1 in \cite{minoids}, $Q_i$ can be made explicit and one can find a change of coordinates $h$ and a gauge $G$ such that $\widehat{\Phi}_{i,t} := (Q_iH)^{-1}\left(h^* \widetilde{\Phi}_{i,t}\right)G$ solves  $d\widehat{\Phi}_{i,t} = \widehat{\Phi}_{i,t}\widehat{\xi}_{i,t}$ where $\widehat{\xi}_{i,t}$ is a perturbed Delaunay potential and $\widehat{\Phi}_{i,0}(z) = z^{A_0}$. One can thus apply Theorem \ref{theoremPerturbedDelaunay} on $\widehat{\xi}_{i,t}$ and $\widehat{\Phi}_{i,t}$, which proves the existence of the family $\left(M_t\right)_{-T<t<T}$ of CMC $H$ surfaces of genus zero and $n$ Delaunay ends, each of weight (according to Equation \eqref{eqPoidsDelaunayPure})
\begin{equation*}
w_{i,t} = 8\pi rs\sinh q = 2\pi t \alpha_{i,t},
\end{equation*}
which proves the first point of Theorem \ref{theoremConstructionMinoids}. Let $\widehat{f}_{i,t}:=\Sym_q\left(\Uni \widehat{\Phi}_{i,t}\right)$ and let $\widehat{f}_{i,t}^\dD$ be the Delaunay immersion given by Theorem \ref{theoremPerturbedDelaunay}.

\paragraph{Limit axis.}

In order to compute the limit axis of $f_t$ at the end around $p_{i,t}$, let $\widehat{\Delta}_{i,t}$ be the oriented axis of $\widehat{f}_{i,t}^\dD$ at $z=0$. Then, using Theorem \ref{theoremPerturbedDelaunay},
\begin{equation*}
\widehat{\Delta}_{i,0} =  \geod\left( \I_2, -\sigma_1 \right).
\end{equation*}
And using $\widehat{f}_{i,t}(z) = (Q_iH)^{-1}\cdot \left( h^*f_t(z) \right)$,
\begin{equation*}
\widehat{\Delta}_{i,0} = (Q_iH)^{-1}\cdot\Delta_{i,0},
\end{equation*}
and thus,
\begin{equation*}
\Delta_{i,0} = (QH) \cdot \geod(\I_2,-\sigma_1).
\end{equation*}
Compute $H\cdot (-\sigma_1) = \sigma_3$ and note that $M_iH^{-1} = \Phi_0(\pi_i)$ to get
\begin{equation*}
	\Delta_{i,0} = \geod\left( \I_2, N_0(p_{i,0}) \right)
\end{equation*}
where $N_0$ is the normal map of the minimal immersion.

\paragraph{Type of the ends.}

Suppose that $t$ is positive. Then the end at $p_{i,t}$ is unduloidal if, and only if its weight is positive; that is, $\alpha_{i,t}$ is positive. Use Proposition 5 of \cite{minoids} to show that if the normal map $N_0$ of $M_0$ points toward the inside, then $\alpha_{i,0}=\tau_i$ where $2\pi\tau_iN_0(p_{i,0})$ is the flux of $M_0$ around the end at $p_{i,0}$ ($\alpha_{i,0}=-\tau_i$ for the other orientation). Thus, if $M_0$ is Alexandrov-embedded, then the ends of $M_t$ are of unduloidal type for $t>0$ and of nodoidal type for $t<0$.

\subsection{Alexandrov-embeddedness}\label{sectionAlexEmbeddedMnoids}

In order to show that $M_t$ is Alexandrov-embedded for $t>0$ small enough, one can follow the proof of Proposition 6 in \cite{minoids}. Note that this proposition does not use the fact that $M_t$ is CMC $H$, but relies on the fact that the ambient space is $\Rr^3$. This enjoins us to lift $f_t$ to $\Rr^3$ via the exponential map at the identity, hence defining an immersion $\widehat{f}_t: \Sigma_t\longrightarrow\Rr^3$ which is not CMC anymore, but is Alexandrov-embedded if, and only if $f_t$ is  Alexandrov-embedded. Let $\psi:\Sigma_0\longrightarrow M_0\subset \Rr^3$ be the limit minimal immersion. In order to adapt the proof of \cite{minoids} and show that $M_t$ is Alexandrov-embedded, one will need the following Lemma.

\begin{lemma}
	Let $\widetilde{f}_t := \frac{1}{t}\widehat{f}_t$. Then $\widetilde{f}_t$ converges to $\psi$ on compact subsets of $\Sigma_0$.
\end{lemma}
\begin{proof}
	For all $z$, $$\exp_{\I_2}(\widehat{f}_0(z)) = f_0(z) = \I_2,$$ so $\widehat{f}_0(z) = 0$. Thus
	\begin{equation*}
		\lim\limits_{t\to 0} \widetilde{f}_t(z) = \frac{d}{dt}\widehat{f}_t(z)\mid_{t=0}.
	\end{equation*}
	Therefore, using Proposition \ref{propBlowupCatenoid},
	\begin{align*}
		\psi(z) &= \lim\limits_{t\to 0} \frac{1}{t}\left( f_t(z) - \I_2 \right)\\
		&= \lim\limits_{t\to 0} \frac{1}{t}\left( \exp_{\I_2}(\widehat{f}_t(z)) - \exp_{\I_2}(\widehat{f}_0(z)) \right)\\
		&= \frac{d}{dt}\exp_{\I_2}(\widehat{f}_t(z))\mid_{t=0}\\
		&= d\exp_{\I_2}(0)\cdot \frac{d}{dt}\widehat{f}_t(z)\mid_{t=0}\\
		&= \lim\limits_{t\to 0} \widetilde{f}_t(z).
	\end{align*}
\end{proof}

\pagebreak
\appendix

\section{CMC surfaces of revolution in $\Hh^3$}
\label{appendixJleliLopez}

Following Sections 2.2 and 2.3 of \cite{jlelilopez},

\begin{proposition}\label{propJleliLopez}
	Let $X:\Rr\times\left[0,2\pi\right]\longrightarrow\Hh^3$ be a conformal immersion of revolution with metric $g^2(s)\left(ds^2+d\theta^2\right)$. If $X$ is CMC $H>1$, then $g$ is periodic and denoting by $S$ its period,
	\begin{equation*}
		\sqrt{H^2-1} \int_{0}^{S} g(s)ds = \pi \quad \text{and} \quad  \int_{0}^{S}\frac{ds}{g(s)} = \frac{2\pi^2}{|w|}
	\end{equation*}
	where $w$ is the weight of $X$, as defined in \cite{loopgroups}.
\end{proposition}
\begin{proof}
	According to Equation (11) in \cite{jlelilopez}, writing $\tau = \frac{\sqrt{|w|}}{\sqrt{2\pi}}$ and $g = \tau e^\sigma$,
	\begin{equation}\label{eqDiffsigma}
		\left(\sigma'\right)^2 = 1-\tau^2\left( \left( He^{\sigma} + \iota e^{-\sigma} \right)^2 - e^{2\sigma} \right)
	\end{equation}
	where $\iota\in\left\{\pm 1\right\}$ is the sign of $w$. The solutions $\sigma$ are periodic with period $S>0$. Apply an isometry and a change of the variable $s\in\Rr$ so that 
	\begin{equation*}
		\sigma'(0)=0 \quad \text{and}\quad \sigma(0) = \min\limits_{s\in\Rr} \sigma(s).
	\end{equation*}
	By symmetry of Equation \eqref{eqDiffsigma}, one can thus define
	\begin{equation*}
		a := e^{2\sigma(0)}= \min\limits_{s\in\Rr} e^{2\sigma(s)} \quad \text{and}\quad b := e^{2\sigma\left(\frac{S}{2}\right)} = \max\limits_{s\in\Rr} e^{2\sigma(s)}.
	\end{equation*}
	With these notations, Equation \eqref{eqDiffsigma} can be written in a factorised form as
	\begin{equation}\label{eqDiffsigmafactorised}
		\left( \sigma' \right)^2 = \tau^2(H^2-1)e^{-2\sigma}\left( b - e^{2\sigma} \right)\left(e^{2\sigma} - a\right)
	\end{equation}
	with
	\begin{equation}\label{eqaPeriode}
		a = \frac{1-2\iota\tau^2 H - \sqrt{1-4\tau^2(\iota H - \tau^2)}}{2\tau^2 (H^2-1)}
	\end{equation}
	and
	\begin{equation*}
		b = \frac{1-2\iota\tau^2 H + \sqrt{1-4\tau^2(\iota H - \tau^2)}}{2\tau^2 (H^2-1)}.
	\end{equation*}
	In order to compute the first integral, change variables $v=e^\sigma$, $y=\sqrt{b-v^2}$ and $x= \frac{y}{\sqrt{b - a}}$ and use Equation \eqref{eqDiffsigmafactorised} to get
	\begin{align*}
		\sqrt{H^2-1} \int_{0}^{S} \tau e^{\sigma(s)}ds &= 2 \sqrt{H^2-1} \int_{\sqrt{a}}^{\sqrt{b}} \frac{\tau v dv}{\tau\sqrt{H^2-1}\sqrt{b-v^2}\sqrt{v^2-a}} \\
		&= -2\int_{\sqrt{b-a}}^{0}\frac{dy}{\sqrt{b-a-y^2}} \\
		&= 2\int_{0}^{1}\frac{dx}{\sqrt{1-x^2}} \\
		&= \pi.
	\end{align*}
	In the same manner with the changes of variables $v=e^{-\sigma}$, $y=\sqrt{a^{-1}-v^2}$ and $x = \frac{y}{\sqrt{a^{-1}-b^{-1}}}$,
	\begin{align*}
		\int_{0}^{S}\frac{ds}{\tau e^{\sigma(s)}} &= \frac{-2}{\tau\sqrt{H^2-1}}  \int_{a^{-1/2}}^{b^{-1/2}}\frac{ dv}{v\sqrt{b-v^{-2}}\sqrt{v^{-2}-a}}\\
		&= \frac{2}{\tau^2\sqrt{H^2-1}}\int_{0}^{\sqrt{a^{-1}-b^{-1}}}\frac{dy}{\sqrt{b-a-aby^2}}\\
		&= \frac{2}{\tau^2\sqrt{H^2-1}\sqrt{ab}}\int_{0}^{1}\frac{dx}{\sqrt{1-x^2}}\\
		&= \frac{\pi}{\tau^2}
	\end{align*}
	because $ab=\frac{1}{H^2-1}$.
\end{proof}

\begin{lemma}\label{lemmaTubularRadius}
	Let $\dD_t$ be a Delaunay surface in $\Hh^3$ of constant mean curvature $H>1$ and weight $2\pi t>0$ with Gauss map $\eta_t$. Let $r_t$ be the maximal value of $R$ such that the map
	\begin{equation*}
		\function{T}{\left(-R,R\right)\times \dD_t}{\Tub_{r_t}\subset\Hh^3}{(r,p)}{\geod(p,\eta_t(p))(r)}
	\end{equation*}
	is a diffeomorphism. Then $r_t\sim t$ as $t$ tends to $0$.
\end{lemma}
\begin{proof}
	The quantity $r_t$ is the inverse of the maximal geodesic curvature of the surface. This maximal curvature is attained for small values of $t$ on the points of minimal distance between the profile curve and the axis. Checking the direction of the mean curvature vector at this point, the maximal curvature curve is not the profile curve but the parallel curve. Hence $r_t$ is the minimal hyperbolic distance between the profile curve and the axis. A study of the profile curve's equation as in Proposition \ref{propJleliLopez} shows that
	\begin{equation*}
	r_t = \sinh^{-1}\left( \tau \exp\left( \sigma_{\min} \right) \right) = \sinh^{-1}\left( \tau \sqrt{a(\tau)} \right).
	\end{equation*}
	But using Equation \eqref{eqaPeriode}, as $\tau$ tends to $0$, $a\sim \tau^2 = |t|$, which gives the expected result.
\end{proof}

\begin{lemma}\label{lemmaCompNormalesDelaunay}
	Let $\dD_t$ be a Delaunay surface in $\Hh^3$ of weight $2\pi t>0$ with Gauss map $\eta_t$ and maximal tubular radius $r_t$.
	There exist $T>0$ and $\alpha<1$ such that for all $0<t<T$ and $p,q\in\dD_t$ satisfying $\distH{p}{q}<\alpha r_t$,
	\begin{equation*}
		\norm{\Gamma_p^q\eta_t(p) - \eta_t(q)}<1.
	\end{equation*}
\end{lemma}
\begin{proof}
	Let $t>0$. Then for all $p,q\in\dD_t$,
	\begin{equation*}
		\norm{\Gamma_p^q\eta_t(p) - \eta_t(q)} \leq \supp{s\in\gamma_t}\norm{II_t(s)}\times \ell(\gamma_t)
	\end{equation*}
	where $II_t$ is the second fundamental form of $\dD_t$, $\gamma_t\subset \dD_t$ is any path joining $p$ to $q$ and $\ell(\gamma_t)$ is the hyperbolic length of $\gamma_t$. Using the fact that the maximal geodesic curvature $\kappa_t$ of $\dD_t$ satisfies $\kappa_t\sim \coth r_t$ as $t$ tends to zero, there exists a uniform constant $C>0$ such that  
	\begin{equation*}
		\supp{s\in\dD_t}\norm{II_t(s)}<C\coth r_t.
	\end{equation*}
	Let $0<\alpha < (1+C)^{-1}<1$ and suppose that $\distH{p}{q}<\alpha r_t$. Let $\sigma_t: [0,1]\rightarrow\Hh^3$ be the geodesic curve of $\Hh^3$ joining $p$ to $q$. Then $\sigma_t([0,1])\subset \Tub_{\alpha r_t}$ and thus the projection $\pi_t: \sigma_t([0,1])\rightarrow \dD_t$ is well-defined. Let $\gamma_t:=\pi_t\circ \sigma_t$. Then
	\begin{align*}
		\norm{\Gamma_p^q\eta_t(p) - \eta_t(q)} &\leq C\coth r_t \times \supp{s\in \sigma_t}\norm{d\pi_t(s)}\\
		&\leq C\coth r_t \times \supp{s\in \Tub_{\alpha r_t}}\norm{d\pi_t(s)}\times \distH{p}{q}\\
		&\leq C\coth r_t \times \frac{\tanh r_t}{\tanh r_t - \tanh(\alpha r_t)}\times \alpha r_t\\
		&\leq \frac{C\alpha r_t}{\tanh r_t - \tanh(\alpha r_t)}\sim \frac{C\alpha}{1-\alpha}<1
	\end{align*}
	as $t$ tends to zero.
\end{proof}

\section{Remarks on the polar decomposition}\label{appendixDecompositionPolaire}

Let $\SLdeuxCplusplus$ be the subset of $\SL(2,\Cc)$ whose elements are hermitian positive definite. Let
\begin{equation*}
	\function{\Pol}{\SL(2,\Cc)}{\SLdeuxCplusplus\times \SU(2)}{A}{\left(\Pol_1(A),\Pol_2(A)\right)}
\end{equation*}
be the polar decomposition on $\SL(2,\Cc)$. This map is differentiable and satisfies the following proposition.
\begin{proposition}\label{propDiffPolar2}
	For all $A\in\SL(2,\Cc)$, $\norm{d\Pol_2\left(A\right)} \leq |A|$.
\end{proposition}
\begin{proof}
	We first write the differential of $\Pol_2$ at the identity in an explicit form. Writing
	\begin{equation*}
		\function{d\Pol_2(\I_2)}{\slfrak(2,\Cc)}{\su(2)}{M}{\pol_2(M)}
	\end{equation*}
	gives
	\begin{equation*}
		\pol_2\begin{pmatrix}
		a & b \\
		c & -a
		\end{pmatrix} = \begin{pmatrix}
		i\Im a & \frac{b-\conj{c}}{2} \\
		\frac{c-\bar{b}}{2} & -i\Im a
		\end{pmatrix}.
	\end{equation*}
	Note that for all $M\in\slfrak(2\Cc)$,
	\begin{align*}
	\left|\pol_2(M)\right|^2 &= 2\left(\Im a\right)^2 + \frac{1}{4}\left( \left| b-\conj{c} \right|^2 + \left|c-\conj{b}\right|^2 \right)\\
	&\leq \left|M\right|^2 - \frac{1}{2}\left|b+\conj{c}\right|^2\\
	&\leq \left|M\right|^2.
	\end{align*}
	
	We then compute the differential of $\Pol_2$ at any point of $\SL(2,\Cc)$. Let $(S_0,Q_0)\in\SLdeuxCplusplus\times \SU(2)$. Consider the differentiable maps
	\begin{equation*}
		\function{\phi}{\SL(2,\Cc)}{\SL(2,\Cc)}{A}{S_0AQ_0} \quad \text{and}\quad \function{\psi}{\SU(2)}{\SU(2)}{Q}{QQ_0.}
	\end{equation*}
	Then $\psi\circ \Pol_2 \circ \phi^{-1} = \Pol_2$ and for all $M\in T_{S_0Q_0}\SL(2,\Cc)$,
	\begin{equation*}
		d\Pol_2\left(S_0Q_0\right)\cdot M = \pol_2\left( S_0^{-1}MQ_0^{-1} \right)Q_0.
	\end{equation*}
	
	Finally, let $A\in\SL(2,\Cc)$ with polar decomposition $\Pol( A) = (S,Q)$. Then for all $M\in T_A\SL(2,\Cc)$,
	\begin{align*}
		\left|d\Pol_2(A)\cdot M\right| = \left|\pol_2\left( S^{-1}MQ^{-1} \right)Q\right|\leq \left|S\right|\times \left|M\right|
	\end{align*}
	and thus using 
	\begin{equation*}
		S = \exp\left( \frac{1}{2}\log\left(AA^*\right) \right)
	\end{equation*}
	gives
	\begin{equation*}
		\norm{d\Pol_2(A)} \leq \left|S\right|\leq \left|A\right|.
	\end{equation*}
\end{proof}

\begin{corollary}\label{corQ2v-Q1v}
	Let $0<q<\log \rho$ and $F_1,F_2\in\LSUdeuxrho$ with unitary parts $Q_i = \Pol_2(F_i(e^{-q}))$. Let $\epsilon>0$ such that
	\begin{equation*}
		\norm{F_2^{-1}F_1-\I_2}_\rho < \epsilon.
	\end{equation*}
	If $\epsilon$ is small enough, then there exists a uniform $C>0$ such that for all $v\in T_{\I_2}\Hh^3$,
	\begin{equation*}
		\norm{Q_2\cdot v - Q_1\cdot v}_{T_{\I_2}\Hh^3}\leq C\norm{F_2}_{\rho}^2\epsilon.
	\end{equation*}
\end{corollary}
\begin{proof}
	Let $v\in T_{\I_2}\Hh^3$ and consider the following differentiable map
	\begin{equation*}
		\function{\phi}{\SU(2)}{T_{\I_2}\Hh^3}{Q}{Q\cdot v.}
	\end{equation*}
	Then
	\begin{align*}
		\norm{Q_2\cdot v - Q_1\cdot v}_{T_{\I_2}\Hh^3} &= \norm{\phi(Q_2) - \phi(Q_1)}_{T_{\I_2}\Hh^3} \\
		&\leq \supp{t\in\left[0,1\right]}\norm{d\phi(\gamma(t))}\times \int_{0}^{1}\left| \dot{\gamma}(t) \right|dt
	\end{align*}
	where $\gamma : \left[0,1\right]\longrightarrow\SU(2)$ is a path joining $Q_2$ to $Q_1$. Recalling that $\SU(2)$ is compact gives
	\begin{equation}\label{eqQ2v-q1v}
		\norm{Q_2\cdot v - Q_1\cdot v}_{T_{\I_2}\Hh^3} \leq C\left|Q_2-Q_1\right|
	\end{equation}
	where $C>0$ is a uniform constant. But writing $A_i=F_i(e^{-q})\in\SL(2,\Cc)$,
	\begin{align*}
		\left|Q_2-Q_1\right| &= \left|\Pol_2(A_2) - \Pol_2(A_1)\right|\\
		&\leq \supp{t\in\left[0,1\right]}\norm{d\Pol_2(\gamma(t))}\times \int_{0}^{1}\left| \dot{\gamma}(t) \right|dt
	\end{align*}
	where $\gamma : \left[0,1\right]\longrightarrow \SL(2,\Cc)$ is a path joining $A_2$ to $A_1$. Take for example
	\begin{equation*}
		\gamma(t) := A_2\times \exp\left( t\log \left( A_2^{-1}A_1 \right) \right).
	\end{equation*}
	Suppose now that $\epsilon$ is small enough for $\log$ to be a diffeomorphism from $D(\I_2,\epsilon)\cap \SL(2,\Cc)$ to $D(0,\epsilon')\cap \slfrak(2,\Cc)$. Then
	\begin{equation*}
		\norm{A_2^{-1}A_1-\I_2}\leq \norm{F_2^{-1}F_1-\I_2}_{\rho}<\epsilon
	\end{equation*}
	implies
	\begin{equation*}
		\left|\gamma(t)\right|\leq \widetilde{C}\left|A_2\right|\quad \text{and}\quad \left|\dot{\gamma}(t)\right|\leq \widetilde{C}\widehat{C}\left|A_2\right|\epsilon
	\end{equation*}
	where $\widetilde{C},\widehat{C}>0$ are uniform constants. Using Proposition \ref{propDiffPolar2} gives
	\begin{equation*}
		\left|Q_2-Q_1\right|\leq \widehat{C}\widetilde{C}^2\left|A_2\right|^2\epsilon
	\end{equation*}
	and inserting this inequality into \eqref{eqQ2v-q1v} gives
	\begin{equation*}
		\norm{Q_2\cdot v - Q_1\cdot v}_{T_{\I_2}\Hh^3} \leq C\widehat{C}\widetilde{C}^2\left|A_2\right|^2\epsilon \leq C\widehat{C}\widetilde{C}^2\norm{F_2}_{\rho}^2\epsilon.
	\end{equation*}
\end{proof}

\pagebreak

\providecommand{\bysame}{\leavevmode\hbox to3em{\hrulefill}\thinspace}
\providecommand{\MR}{\relax\ifhmode\unskip\space\fi MR }
\providecommand{\MRhref}[2]{%
	\href{http://www.ams.org/mathscinet-getitem?mr=#1}{#2}
}
\providecommand{\href}[2]{#2}

\noindent
Thomas Raujouan\\
Institut Denis Poisson\\
Universit\'e de Tours, 37200 Tours, France\\
\verb$raujouan@univ-tours.fr$

\end{document}